\documentclass[twoside,12pt,leqno]{amsart}
\usepackage{amssymb,latexsym,verbatim,enumerate} 
\usepackage{tikz}
\usepackage{hyperref}
\usepackage{amsrefs}
\hypersetup{citecolor=red, linkcolor=blue, colorlinks=true}
\usepackage{booktabs}
\usepackage{nicefrac,stmaryrd} 
\usepackage{mathtools} 

\usepackage{etoolbox} 
\patchcmd{\section}{\scshape}{\bfseries}{}{}
\makeatletter
\patchcmd{\@settitle}{\uppercasenonmath\@title}{}{}{}
\patchcmd{\@setauthors}{\MakeUppercase}{}{}{}
\renewcommand{\@secnumfont}{\bfseries} 
\makeatother          

\newtheorem{theorem}{Theorem}[section]

\newtheorem{proposition}[theorem]{Proposition}
\newtheorem{lemma}[theorem]{Lemma}
\newtheorem{corollary}[theorem]{Corollary}
\theoremstyle{definition}
\newtheorem{definition}[theorem]{Definition}
\newtheorem{remark}[theorem]{Remark}

\newtheorem{problem}[theorem]{Problem}

\newcommand{\im}{\textup{im}}
\renewcommand{\ge}{\geqslant}
\renewcommand{\le}{\leqslant}
\newcommand{\lhdeq}{\trianglelefteqslant}    
\newcommand{\coloneq}{\vcentcolon=}      

\newcommand{\bX}{{\mathbf {X}}}
\newcommand{\bO}{{\mathbf {O}}}
\newcommand{\bU}{{\mathbf {U}}}
\newcommand{\bSp}{{\mathbf {Sp}}}
\newcommand{\cW}{{\mathcal {W}}}

\newcommand{\ABinom}[2]{\genfrac{\langle}{\rangle}{0pt}{}{#1}{#2}_{\kern -2pt q}}
\newcommand{\binomq}[2]{{\binom{#1}{#2}}_{\kern-3pt q}}
\newcommand{\binomqq}[2]{{\binom{#1}{#2}}_{\kern-3pt -q}}

\newcommand{\R}{\mathbb{R}}
\newcommand{\F}{{\mathbb F}}
\newcommand{\Z}{{\mathbb Z}}

\newcommand{\OmX}{\Omega{\mathrm {X}}}

\newcommand{\GX}{{\mathrm {GX}}}

\newcommand{\GL}{{\mathrm {GL}}}

\newcommand{\SL}{{\mathrm {SL}}}
\newcommand{\Sp}{{\mathrm {Sp}}}
\newcommand{\GU}{{\mathrm {GU}}}
\newcommand{\SU}{{\mathrm {SU}}}

\newcommand{\GO}{{\mathrm {GO}}}

\newcommand{\rhosting}{\rho_{{\rm duo}}^{\rm stingray}}

\newcommand{\eps}{\varepsilon}

\newcommand\cC{\mathcal{C}}
\newcommand\cU{\mathcal{U}}
\newcommand\cX{\mathcal{X}}

\newcommand{\Binom}[2]{{\genfrac{[}{]}{0pt}{}{#1}{#2}}}

\newcommand{\oPerpSymbol}{\begin{tikzpicture}[scale=0.134]  
  \draw (0,-0.5)--(0,1); \draw (-0.866,-0.5)--(0.866,-0.5);
  \draw (0,0) circle [radius=1];\end{tikzpicture}}
\newcommand{\oPerp}{\mathbin{\raisebox{-1pt}{\oPerpSymbol}}}

\title[Random generation of direct sums of finite non-degenerate subspaces]{\Large Random generation of direct sums of\\ finite non-degenerate subspaces}
\author{S.\,P. Glasby}
\author{Alice C.~Niemeyer}
\author{Cheryl E. Praeger}
\address[Alice C.~Niemeyer]{
Algebra and Representation Theory, RWTH Aachen University,
Pontdriesch 10-16, 52062 Aachen, Germany.
Email: {\tt alice.niemeyer@mathb.rwth-aachen.de}}

\address[S.\,P. Glasby and Cheryl E. Praeger]{Centre for the Mathematics of Symmetry and Computation, University of Western Australia, 35 Stirling Highway, Perth 6009, Australia. E-mail: {\tt\texttt{stephen.glasby@uwa.edu.au, cheryl.praeger@uwa.edu.au}}}
\thanks{\hskip-4mm Acknowledgements: The authors gratefully acknowledge support
  from the Australian Research Council (ARC) Discovery Project DP190100450.
  ACN acknowledges that this is a contribution to
  Project-ID 286237555 - TRR 195 - by the Deutsche Forschungsgemeinschaft
  (DFG, German Research Foundation). We thank the referee for suggesting shorter proofs of Lemmas 3.7(a), 3.8(a) and 3.11.\newline
  2010 Math Subject Classification: 20F65, 20D05, 05-08, 20D06, 68W20.\hfill
  Date: \today.}

\begin{document}

\begin{abstract}
Let $V$ be a $d$-dimensional vector space over a finite field $\F$
equipped with a non-degenerate hermitian, alternating, or quadratic form.
Suppose $|\F|=q^2$ if $V$ is hermitian, and $|\F|=q$ otherwise.
Given integers $e, e'$ such that $e+e'\le d$, we estimate the
proportion of pairs $(U, U')$, where $U$ is  a non-degenerate
$e$-subspace of $V$ and $U'$ is a non-degenerate $e'$-subspace  of $V$,
such that $U\cap U'=0$ and $U\oplus U'$ is non-degenerate (the sum
$U\oplus U'$ is direct and usually not perpendicular). The 
proportion is shown to be positive and at least $1-c/q>0$ for some
constant $c$.
For example, $c=7/4$ suffices in both the unitary and
symplectic cases.
The arguments in the orthogonal case are delicate and assume
that $\dim(U)$ and $\dim(U')$ are even, an assumption relevant for
an algorithmic application (which we discuss) for recognising finite
classical groups.
We also describe how recognising a classical groups $G$ relies on
a connection between certain pairs $(U,U')$ of non-degenerate subspaces
and certain pairs $(g,g')\in G^2$ of group elements
where $U=\textup{im}(g-1)$ and $U'=\textup{im}(g'-1)$.
\end{abstract}


\maketitle
\vskip-8mm
{\SMALL\textsc{Dedication:}  To the memory of Joachim Neub\"user, a pioneer in
Computational Group Theory who envisioned \phantom{|}\hskip 21mm the free software system {\sc GAP}.
}\newline
\phantom{|}\hskip0mm
{\SMALL \textsc{Keywords:} non-degenerate, direct sum, finite classical group, proportion}

\section{Introduction}

Let $V=\F^d$ be a vector space over a finite field $\F$ endowed with
a non-degenerate symplectic, unitary or quadratic form. Given
positive integers $e, e'$ such that $e+e'\le d$, for which non-degenerate
subspaces of dimensions $e,e'$ exist, we show that a
constant proportion of ordered pairs of non-degenerate subspaces of
dimensions $e, e'$, respectively, span a
non-degenerate $(e+e')$-subspace of $V$. Moreover, we prove that the
proportion approaches $1$ when~$|\F|$ approaches infinity.


\begin{theorem}\label{t:main1}
  Let $V=\F^d$ be a vector space over a finite field $\,\F$,
  as in Table~$\ref{tab1}$, equipped with
  a non-degenerate symplectic, unitary or quadratic form. Let
  $e, e'$ be positive integers such that $e+e'\le d$, and let $c$
  be a constant, with the type of form, $q$, $e,e'$ and $c$ as in
  Table~$\ref{tab1}$.  Then the proportion of pairs $(U, U')$ of
  non-degenerate subspaces of dimensions $e, e'$ respectively
  $($\kern-1pt of fixed but arbitrary type $\sigma,\sigma'\in\{-,+\}$
  in the orthogonal case$)$ that
  span a non-degenerate $(e+e')$-subspace $($\kern-1pt of arbitrary type in the orthogonal case$)$ is at least $1-c/q>0$.
\end{theorem}

\begin{center}
\begin{table}[!ht]\label{tab1}
\caption{The form, the field, $q$, and the constant $c$ for Theorem~\ref{t:main1}.}
\begin{tabular}{rllllll}  \hline
$\bX$&&Case	&Form  & $|\F|$ & $c$ & Conditions \\ \hline
$\bU$&&unitary &hermitian & $q^2$ &  $1.72$ & $q\ge2$, $e,e'\ge1$ \\ 
$\bSp$&&symplectic &alternating & $q$ & $1.75$ &  $q\ge2$, $e,e'\ge2$ even \\ 
$\bO^\eps$&&orthogonal & quadratic & $q$   &$3.125$ &  $q\ge4$, $e,e'\ge2$ even\\
&&			&&&$2.85$& $q=3$, $e,e'\ge2$ even\\ \hline
\end{tabular} 
\end{table}
\end{center}

The condition in the orthogonal case that $e,e'$ are both even
not only simplifies
our proof, but it is precisely the assumption that we require
for an algorithmic application for recognising finite classical groups, see
Section~\ref{sub:motiv}.
In the orthogonal case our methods are not strong enough to prove
the result with a sufficiently small value of the constant $c$ when $|\F|=2$.
Extensive computer experimentation suggests that the result also holds
in the case and when $e,e'$ may not both be even, and also when $q=2$.

\begin{problem}
Show that Theorem~$\ref{t:main1}$ holds also in the orthogonal case with $|\F|=2$ (for some constant $c<2$). 
\end{problem}

Our strategy for proving Theorem~\ref{t:main1} is described in
Section~\ref{s:setup}. It allows us to reduce to the case $d = e + e'$,
where, after some delicate analysis, we can apply~\cite{Forms}*{Theorem~1.1}
(which proves Theorem~\ref{t:main1} in the case $d = e + e'$).
We paraphrase this result in the symplectic and unitary cases
in Theorem~\ref{T:Forms} and in the orthogonal case in~Theorem~\ref{T:O}.
When the summands of a direct sum are perpendicular we use
the symbol $\oPerp$ instead of~$\oplus$.

\subsection{Algorithmic motivation}\label{sub:motiv}
Motivation  for  proving  Theorem~\ref{t:main1},  for  us,  came  from
computational group theory, where such non-degenerate subspaces $U, U'$
are  constructed   as  part  of  various   randomised  algorithms  for
recognising finite classical groups. To justify these algorithms it is
necessary to find  a lower bound for the probability  that $U, U'$ are
disjoint and span a non-degenerate  subspace, and this is precisely
what   Theorem~\ref{t:main1}  does.    Thus  our   main  interest   in
Theorem~\ref{t:main1} was  to justify  new algorithms which  we are
developing  for  classical  groups  over finite  fields  of  arbitrary
characteristic~\cite{GLNP}.

Moreover, in the course of our  research we discovered that justifying a
related  probability  bound  in   the  analysis  of  the  constructive
recognition algorithm in \cite{DLLO13} was  overlooked in the cases of
unitary, symplectic  and orthogonal groups in  even characteristic, as
the proof was given only for groups containing $\SL(V)$. In
Subsection~\ref{s:error}  we show  how to  use Theorem~\ref{t:main1}  to
complete the  analysis of  the algorithm  in \cite{DLLO13},  namely to
complete the proof of \cite{DLLO13}*{Lemma 5.8} and hence of the
crucial result~\cite{DLLO13}*{Lemma 5.8}.

Our new algorithms, and also the algorithms in \cite{DLLO13}, work
with a finite classical group $G$ with natural module $V$.  A major
step is to construct a subgroup which is itself a classical group
acting on a much smaller subspace.  The basic strategy is to find two
random elements $g$ and $g'$ in $G$ with the following property: $g$
preserves a decomposition $V= U_g\oPerp F_g$ where $g$ is
irreducible on $U_g=\textup{im}(g-1)$, $F_g=\ker(g-1)$, and
similarly for $g'$. It is convenient to write $U=U_g, F=F_g, U'=U_{g'}$
and $F'=F_{g'}$.  In the case of unitary, symplectic and orthogonal
groups (which we will informally refer to as the \emph{classical
case}), the subspaces $U, U'$ and $F,F'$ are non-degenerate.  The
challenge in all cases is to prove that with high probability the subspaces
$U,U'$ are disjoint, $U+U'$ is non-degenerate, and the subgroup
$\langle g, g' \rangle$ induces a classical group on $U+U'$. This problem can
be subdivided into three smaller problems. The first problem is
addressed by the results of this paper, namely to estimate the
probability that the subspaces are disjoint and span a
non-degenerate subspace.  Second, is the link between the proportion of
pairs of elements $(g, g')$ and the proportion of pairs of suitable
subspaces $(U,U')$.  We comment on this in Subsection~\ref{subsec:stingelts}. The
third problem, given an $(e+e')$-subspace $W$, which is non-degenerate
in the classical case, is to estimate the probability that a pair $(g,g')$
in the relevant classical group on $W$ corresponds to a suitable
pair $(U, U')$ spanning $W$ such that the subgroup $\langle g, g'\rangle$
induces a
classical group on $W$. Resolving this last problem requires deep
theory relying on the finite simple group classification. In the
special case where $e=e'$ this problem has been solved in
\cite{PSY2015}, and we plan to tackle the general case in
collaboration also with L\"ubeck in \cite{GLNP}.

  The strategy for proving Theorem~\ref{t:main1}, the
  notation $\Binom{V}{e}$ and $\Binom{V}{e}^\eps_\sigma$, and the links
  with~\cite{DLLO13}, are described
  in Section~\ref{sec:dllo}. The algorithmic applications are discussed in more detail in Section~\ref{sec:algo}. Formulas for
  the number of non-degenerate $e$-subspaces are given in
  Section~\ref{sec:nondegsub} and certain rational functions in~$q$ such as
  $\omega(d,q)=\prod_{i=1}^d(1-q^{-i})$
  are bounded in Section~\ref{sec:nt}.
  Finally, the proof of Theorem~\ref{t:main1} is detailed
  in Section~\ref{sec:proofs} for the symplectic,
  unitary and orthogonal cases.

\section{Strategy for proving Theorem~\ref{t:main1} and links with \texorpdfstring{\cite{DLLO13}}{}}\label{sec:dllo}

We introduce the notation we will use throughout the paper in
Subsection~\ref{s:notation}, and we explain in
our strategy for proving
Theorem~\ref{t:main1}, which allows us to build on the work in~\cite{Forms},
in Subsection~\ref{s:setup}. More details are given in
Subsection~\ref{s:arrogant} on the algorithmic application discussed in
Subsection~\ref{sub:motiv}, in particular we prove~\cite{DLLO13}*{Lemma~5.8}
for the classical groups.

\subsection{Notation and hypotheses}\label{s:notation}
Let $V=\F^d$ be a $d$-dimensional classical space.

(a) Suppose that $V$ admits a non-degenerate form of type $\bX$ and
$\F$ is a finite field of prime power order $q$ or $q^2$ as in
Table~\ref{tab1}. In particular, if $\bX=\bO^\eps$ then either $d$ is
even with $q>2$ and $\eps\in\{-,+\}$, or $dq$ is odd and $\eps=\circ$,
so the polar form of the quadratic form on $V$ is a non-degenerate
symmetric form. In the orthogonal case we sometimes simply write
$\bX=\bO$. In all cases $V$ is equipped with a non-degenerate (bi- or
sesqui-)linear form.

(b) Let $G=G^\bX$ be a group of isometries of $V$ satisfying
$\OmX_d(q) \lhdeq G \le  \GX_d(q)$. Hence
\begin{enumerate}
\item[(i)] if $\bX=\bSp$ then  $\Sp_d(q) \lhdeq G \le  \textup{GSp}_d(q)$;
\item[(ii)] if $\bX=\bU$ then  $\SU_d(q) \lhdeq G \le  \GU_d(q)$ where
  $\GU_d(q)$ is a subgroup of $\GL_d(q^2)$;
\item[(iii)] if $\bX=\bO^\eps$ then $\Omega^\eps_d(q) \lhdeq G \le \GO^\eps_d(q)$
  with $\eps\in\{-,\circ,+\}$.
\end{enumerate}

(c) An  $e$-dimensional subspace of $V$ is called an
\emph{$e$-subspace} and the set of all of non-degenerate $e$-subspaces
is denoted by $\Binom{V}{e}^\bX$. Therefore $V=U\oPerp U^\perp$ for
each $U\in\Binom{V}{e}^\bX$. If $\bX=\bSp$ or $\bU$ then the subgroup
$\OmX_d(q)$ (and hence $G^\bX$) is transitive on $\Binom{V}{e}^\bX$
(whenever the set is non-empty). On the other hand, if $\bX=\bO^\eps$
then we sometimes write $\Binom{V}{e}^\bX$ as $\Binom{V}{e}^\bO$ or
$\Binom{V}{e}^\eps$. Now $G=G^\bX$ has at most two orbits on
$\Binom{V}{e}^\eps$. As $e$ is even, there are exactly two $G$-orbits
for each $G$ satisfying $\Omega_d^\eps(q)\lhdeq G\le \GO_d^\eps(q)$,
namely the set ${\Binom{V}{e}}^\eps_\sigma$ of $e$-subspaces of type
$\sigma$, where $\sigma\in\{-,+\}$. Thus $\Binom{V}{e}^\eps =
{\Binom{V}{e}}^\eps_+\dot\cup{\Binom{V}{e}}^\eps_-$. Furthermore, if
$U\in\Binom{V}{e}^\eps_\sigma$, then since $V=U\oPerp U^\perp$, it
follows from~\cite{KL}*{Proposition~2.5.11(ii)} that $U^\perp$ has
type~$\eps\sigma$ (also when $d$ is odd, in which case we interpret
$\circ\sigma$ as $\circ$).


(d) Let $e, e'$ be positive integers such that $e+e'\le d$, as in
Table~\ref{tab1}. Now $G^\bX$ acts on $\Binom{V}{e}^\bX$ and $\Binom{V}{e'}^\bX$.
Let $\cU\subseteq\Binom{V}{e}^\bX$ and $\cU'\subseteq\Binom{V}{e'}^\bX$ be
$G^\bX$-orbits. As discussed in (c), if  $\bX=\bSp$ or~$\bU$, then we must
have $\cU=\Binom{V}{e}^\bX$ and $\cU'=\Binom{V}{e'}^\bX$, since these sets
are $G^\bX$-orbits. On the other hand if $\bX=\bO^\eps$, then
$\cU=\Binom{V}{e}^\eps_\sigma$ and $\cU'=\Binom{V}{e'}^\eps_{\sigma'}$ are proper
subsets of $\Binom{V}{e}^\eps$ and $\Binom{V}{e'}^\eps$ where each $U\in\cU$
and $U'\in\cU'$ has subspace type $\sigma,\sigma'\in\{-,+\}$, respectively.
%
We call a pair of subspaces $(U,U')\in\cU\times \cU'$ an
\[
\textup{
\emph{$\bX$-duo} (or a \emph{subspace-duo}) if $U\cap U'=0$ and $U\oplus U'$ is
non-degenerate.}
\]
Hence $\dim(U\oplus U')=e+e'$ and $U\oplus U'\in \Binom{V}{e+e'}^\bX$.  We
use `subspace-duo' instead of `$\bX$-duo' only when the value of $\bX$
is unambiguous.  The proportion we need to estimate in
Theorem~\ref{t:main1} is therefore
\begin{equation}\label{d:rho}
\rho(\bX, V, \cU, \cU') \coloneq \frac{ |\{\, \mbox{$\bX$-duos in\ } \cU\times\cU' \}|}{|\cU|\cdot|\cU'|}
\end{equation}
for appropriate $G^\bX$-orbits $\cU$ and $\cU'$. 

\subsection{Strategy for the Proof of Theorem~\ref{t:main1}}\label{s:setup}

Here we explain our strategy for proving Theorem~\ref{t:main1}.  Let
$V=\F^d$ and $G=G^\bX$ be as in Subsection~\ref{s:notation} of type
$\bX$, and let $\cU, \cU'$ be $G$-orbits of subspaces such that
$\cU\subseteq \Binom{V}{e}^\bX$ and $\cU'\subseteq \Binom{V}{e'}^\bX$,
as in Subsection~\ref{s:notation}(d), so that to prove
Theorem~\ref{t:main1} we need to find a lower bound of the form $c/|\F|$ for the
proportion $\rho(\bX, V, \cU, \cU')$ of $\bX$-duos in $\cU\times\cU'$.

If $d=\dim(V)=e+e'$ then an $\bX$-duo spans $V$, and the proportion
$\rho(\bX, V, \cU, \cU')$ is estimated in~\cite{Forms}*{Theorem 1.1}
where the lower bound $c/|\F|$ with $c$ as in \cite{Forms}*{Table~\ref{tab1}}
is established.
Assume henceforth that $d>e+e'$.  In Proposition~\ref{p:reduce} we
give a strategy for reducing the general case to the case of dimension
$e+e'$.

Clearly the set of $\bX$-duos in $\cU\times\cU'$ is $G$-invariant and,
moreover, the group $G$ acts transitively on $\Binom{V}{e+e'}^\bX$ if
$\bX$ is $\bU$ or $\bSp$, while if $\bX=\bO^\eps$, then $G$ has two
orbits on $\Binom{V}{e+e'}^\eps$, namely ${\Binom{V}{e+e'}}^\eps_+$
and ${\Binom{V}{e+e'}}^\eps_-$.
For a given $W\in \Binom{V}{e+e'}^\bX$, the number of $\bX$-duos $(U,
U')$ in $\cU\times\cU'$ such that $W=U\oplus U'$ depends only on the
$G$-orbit containing $W$.  Moreover if $(U, U')$ is an $\bX$-duo in
$\cU\times\cU'$ which spans $W$, then $U\in \Binom{W}{e}^\bX$ and
$U'\in \Binom{W}{e'}^\bX$ so that $(U, U')$ is an $\bX$-duo in
$\Binom{W}{e}^\bX\times \Binom{W}{e'}^\bX$. This provides a critical
link between the proportion $\rho(\bX, V, \cU, \cU')$ we need to
estimate for Theorem~\ref{t:main1} and the proportion $\rho(\bX, W,
\cW, \cW')$ for the smaller space $W$ and $G_W$-orbits $\cW, \cW'$ in
$\Binom{W}{e}^\bX, \Binom{W}{e'}^\bX$, respectively which as we mentioned,
is estimated in~\cite{Forms}*{Theorem 1.1}.

\begin{remark}
  Suppose $V=W$ is a non-degenerate 4-dimensional orthogonal space of
  type~$\tau\in\{-,+\}$, and $V=U\oplus U'$ is a direct sum of
  two non-degenerate 2-subspaces. If $U'$ has type~$\sigma'$ then, 
  perhaps surprisingly, we can say nothing about the type~$\sigma$ of
  a non-degenerate complement~$U$: 
  the type $\sigma$ of $U$ can be~$+$ or~$-$, and
  for large~$q$ each possibility occurs about half the time!
  This does not contradict
  \cite{KL}*{Proposition~2.5.11(ii)} because the sum need not be perpendicular.
  We stress in the above strategy, that there is \emph{no correlation} between
  the subspace type of~$W$ and the subspace types of the orbits $\cW$
  and~$\cW'$ (which are determined by the subspace types $\sigma, \sigma'$
  of $U\in\cU$ and $U'\in\cU'$, respectively).
  Thus in Proposition~\ref{p:reduce}(b) below,  $\tau,\sigma,\sigma'$ are
  independent of each other.
\end{remark}

\begin{proposition}\label{p:reduce}
  Suppose that  $V, G=G^\bX$ are as in Subsection~$\ref{s:notation}$
  of type $\bX$, that $d, e, e'$ are as in Theorem~$\ref{t:main1}$ with
  $d>e+e'$, and that $\cU, \cU'$ are $G$-orbits of subspaces in
  $\Binom{V}{e}^\bX, \Binom{V}{e'}^\bX$, respectively.
  \begin{enumerate}[{\rm (a)}]
  \item If  $\bX$ is $\bU$ or $\bSp$, and $W\in \Binom{V}{e+e'}^\bX$,
    then $\cU=\Binom{V}{e}^\bX,\ \cU'=\Binom{V}{e'}^\bX$, and
  \begin{equation*}
  \rho(\bX, V, \cU, \cU')
  = \frac{|\Binom{V}{e+e'}^\bX|\cdot |\Binom{W}{e}^\bX|\cdot |\Binom{W}{e'}^\bX|}
         {|\Binom{V}{e}^\bX|\cdot |\Binom{V}{e'}^\bX|}
         \cdot \rho\left(\bX, W, \Binom{W}{e}^\bX,\Binom{W}{e'}^\bX\right).
  \end{equation*}

  \item If  $\bX$ is $\bO^\eps$ for some $\eps\in\{-,\circ,+\}$,
  and $\cU= \Binom{V}{e}^\eps_\sigma$, $\cU'= \Binom{V}{e'}^\eps_{\sigma'}$,
  where $e,e'$ are even and $\sigma, \sigma'\in\{-,+\}$, then, choosing
  $W_\tau\in {\Binom{V}{e+e'}}^\eps_\tau$ for $\tau\in\{-,+\}$, we have
  \begin{equation*}
  \rho(\bO^\eps, V, \cU, \cU')=
  \sum_{\tau\in\{-,+\}} \frac{|{\Binom{V}{e+e'}}^\eps_\tau|\cdot
  |{\Binom{W_\tau}{e}}^\tau_\sigma|\cdot |{\Binom{W_\tau}{e'}}^\tau_{\sigma'}|}
    {|{\Binom{V}{e}}^\eps_\sigma|\cdot
      |{\Binom{V}{e'}}^\eps_{\sigma'}|}\cdot
    \rho\left(\bO^\tau, W_\tau, \Binom{W_\tau}{e}^\tau_\sigma, \Binom{W_\tau}{e'}^\tau_{\sigma'}\right).
  \end{equation*}
  \end{enumerate}
\end{proposition}

\begin{proof}
(a) Suppose that $\bX$ is $\bU$ or $\bSp$. As discussed above,
  $\cU=\Binom{V}{e}^\bX,\ \cU'=\Binom{V}{e'}^\bX$, and $G$ is
  transitive on $\Binom{V}{e+e'}^\bX$. Thus each $W\in
  \Binom{V}{e+e'}^\bX$ is spanned by the same number of $\bX$-duos in
  $\cU\times \cU'$, and this number is equal to the number of
  $\bX$-duos in $\Binom{W}{e}^\bX\times\Binom{W}{e'}^\bX$. Thus, for a
  chosen $W\in \Binom{V}{e+e'}^\bX$,
\begin{align*}
\rho(\bX, V, \cU, \cU') &= \frac{|\Binom{V}{e+e'}^\bX|\cdot |\{\,
  \mbox{$\bX$-duos in\ }
  \Binom{W}{e}^\bX\times\Binom{W}{e'}^\bX\}|}{|\Binom{V}{e}^\bX|\cdot
  |\Binom{V}{e'}^\bX|}\\ &= \frac{|\Binom{V}{e+e'}^\bX|\cdot
  |\Binom{W}{e}^\bX|\cdot |\Binom{W}{e'}^\bX|}{|\Binom{V}{e}^\bX|\cdot
  |\Binom{V}{e'}^\bX|}\cdot \rho\left(\bX, W, \Binom{W}{e}^\bX,
\Binom{W}{e'}^\bX\right).
\end{align*}

(b) Now suppose that $\bX$ is $\bO^\eps$ and $\cU=
\Binom{V}{e}^\eps_\sigma$, $\cU'= \Binom{V}{e'}^\eps_{\sigma'}$, for
some $\sigma, \sigma'\in\{-,+\}$. The $G$-orbits on
${\Binom{V}{e+e'}}^\eps$ are ${\Binom{V}{e+e'}}^\eps_+$ and
${\Binom{V}{e+e'}}^\eps_-$, and the number of $\bX$-duos in $\cU\times
\cU'$ spanning a subspace $W\in {\Binom{V}{e+e'}}^\eps$ depends only
on the $G$-orbit containing $W$. Moreover, if $W$ has type
$\tau\in\{-,+\}$, then this number is equal to the number of
$\bX$-duos in
$\Binom{W}{e}^\tau_\sigma\times\Binom{W}{e'}^\tau_{\sigma'}$. Thus,
choosing $W_\tau\in {\Binom{V}{e+e'}}^\eps_\tau$ for $\tau\in\{-,+\}$,
each $\bO^\eps$-duo $(U,U')\in\cU\times\cU'$ such that $U\oplus U'=W_\tau$
is an $\bO^\tau$-duo in $\Binom{W}{e}^\tau_\sigma\times\Binom{W}{e'}^\tau_{\sigma'}$
where $W=W_\tau$,
and conversely each $\bO^\tau$-duo
$(U,U')\in\Binom{W}{e}^\tau_\sigma\times\Binom{W}{e'}^\tau_{\sigma'}$ satisfies
$(U,U')\in\Binom{V}{e}^\eps_\sigma\times\Binom{V}{e'}^\eps_{\sigma'}=\cU\times\cU'$.
Hence 
\begin{align*}
\rho(\bO^\eps, V, \cU, \cU')
	& =   \sum_{\tau\in\{-,+\}} \frac{|{\Binom{V}{e+e'}}^\eps_\tau|\cdot 
|\{\, \mbox{$\bO^\tau$-duos in\ }  \Binom{W_\tau}{e}^\tau_\sigma\times\Binom{W_\tau}{e'}^\tau_{\sigma'}\}|}{|{\Binom{V}{e}}^\eps_\sigma|\cdot |{\Binom{V}{e'}}^\eps_{\sigma'}|}\\
	&=  \sum_{\tau\in\{-,+\}} \frac{|{\Binom{V}{e+e'}}^\eps_\tau|\cdot |{\Binom{W_\tau}{e}}^\tau_\sigma|\cdot |{\Binom{W_\tau}{e'}}^\tau_{\sigma'}|}{|{\Binom{V}{e}}^\eps_\sigma|\cdot |{\Binom{V}{e'}}^\eps_{\sigma'}|}\cdot \rho\left(\bO^\tau, W_\tau, \Binom{W_\tau}{e}^\tau_\sigma, \Binom{W_\tau}{e'}^\tau_{\sigma'}\right).\qedhere
\end{align*}
\end{proof}

\subsection{Links with \texorpdfstring{\cite{DLLO13}}{}}\label{s:arrogant}

The  quantity $\rho(\bX,V,\cU, \cU')$ defined in~\eqref{d:rho}
is equal to the proportion
studied in Theorem~\ref{t:main1} for appropriate subspace families $\cU,\cU'$.
It is often  convenient to  count single  subspaces rather  than subspace
pairs, so we note the following easily proved property.

\begin{lemma}\label{l:dllo2}
  Let $\cU,\cU'$ be $G^\bX$-orbits of subspaces as in
  Proposition~$\ref{p:reduce}$, and let $U\in \cU$. Then 
  \[
    \rho(\bX,V,\cU, \cU')
    = \frac{|\{ U'\in\cU'\mid (U,U')\ \mbox{is an $\bX$-duo} \}|}{|\cU'|}.
  \]
\end{lemma}

\begin{proof}
  Since $\cU$ is an orbit  under each $G^\bX$ satisfying
  $\OmX(V)\lhdeq G^\bX\le\GX(V)$, 
  the number, say $n$, of $U'\in\cU'$ such that $(U,U')$ is an $\bX$-duo
  (that is,  $U+U'$ is non-degenerate of dimension $e+e'$) is independent 
  of the choice of  $U\in\cU$. Hence  we conclude that
  $\displaystyle{\rho(\bX,V,\cU, \cU')
    = \frac{n\cdot |\cU|}{|\cU|\cdot|\cU'|}= \frac{n}{|\cU'|}.}$
\end{proof}

\section{Algorithmic applications of Theorem~\texorpdfstring{\ref{t:main1}}{}}\label{sec:algo}

In this section we describe how the main results of this paper will
be used in an algorithmic context for recognising classical groups.
Conceptually, we wish to construct classical groups of smaller dimension
in a given classical group by constructing a subspace duo $(U,U')$ from
a pair of elements $(g,g')$ which we call a `stingray duo', and which turns out to
generate a classical group on $U+U'$ with high probability.

\subsection{Completing the proof of \cite{DLLO13}*{Lemma 5.8}}\label{s:error}
First we present several results leading up to Lemma~\ref{l:ddlo4},
which deals with~\cite{DLLO13}*{Lemma 5.8}.
Our approach is more general as the two subspaces we treat may come
from different $G^\bX$-orbits. However we attempt, as far as possible,
to use the same notation as in \cite{DLLO13} for clarity.

For a subspace $W$ of $V$,
let $W^\perp=\{v\in V\mid \beta(v,W)=0\}$ where $\beta\colon V\times V\to\F$
is the sesquilinear form preserved by $V$. When a direct sum $U\oplus W$ is a
\emph{perpendicular direct sum}, that is when $U\cap W=0$ and
$U\subseteq W^\perp$, we write $U\oPerp W$ for emphasis.

\begin{lemma}\label{l:dllo1}
  Let $\cU,\cU'$ be $G^\bX$-orbits of subspaces as in
  Proposition~$\ref{p:reduce}$, let  $U\in \cU$, $U'\in\cU'$, and let $E=U^\perp\cap (U')^\perp$ and $W=U + U'$.
  Then
  \begin{enumerate}[{\rm (a)}]
  \item $E=W^\perp$;
  \item $\dim(E)=d-e-e'$ if and only if $W=U\oplus U'$;
  \item $W$ is non-degenerate of dimension $e+e'$ if and only if
    $V=E\oPerp (U\oplus U')$.
  \end{enumerate}
\end{lemma}

\begin{proof}
  (a)~By definition $E= U^\perp\cap (U')^\perp = (U+U')^\perp = W^\perp$.

  (b)~By part~(a), $\dim(E) = d-\dim(W)$. Hence $\dim(E) = d-e-e'$
  if and only if $\dim(W)=e+e'$, which, in turn, is
  equivalent to $W=U\oplus U'$.

  (c)~If $W$ is non-degenerate of dimension $e+e'$, then $W\cap W^\perp=0$ and
  $W=U\oplus U'$. Hence by part~(a),   $V=E\oPerp (U\oplus U')$. Conversely
  if  $V=E\oPerp (U\oplus U')$ then $W=U\oplus U'$ has dimension $e+e'$
  and, by part~(a), $W\cap W^\perp = W\cap E=0$ holds, so $W$ is non-degenerate.
\end{proof}

Lemmas~\ref{l:dllo2} and~\ref{l:dllo1}(c) have the following
immediate corollary.

\begin{corollary}\label{c:dllo1}
  Let $\cU,\cU'$ be $G^\bX$-orbits of subspaces as in
  Proposition~$\ref{p:reduce}$, let  $U\in \cU$, $U'\in\cU'$,
  and let $E= U^\perp\cap (U')^\perp$,
  as in Lemma~$\ref{l:dllo1}$. Then
  \begin{align*}
  \rho(\bX,V,\cU, \cU')
  &= \frac{|\{(U, U')\in\cU\times\cU'\mid V
    =E\oPerp (U\oplus U') \}|}{|\cU|\cdot|\cU'|}\\
  &= \frac{|\{ U'\in\cU'\mid V=E\oPerp (U\oplus U') \}|}{|\cU'|}.
  \end{align*}
\end{corollary}

Lemma 5.8 of \cite{DLLO13} counts group elements rather than subspaces.
Further, our proof applies for all fields $\F$ in the symplectic and unitary
cases, and for $|\F|\ge3$ in the orthogonal case.
Our next result is a more general version of what is required for \cite{DLLO13}
because we do not assume that $\cU=\cU'$.

\begin{lemma}\label{l:dllo3}
  Let $\cU,\cU'$ be $G^\bX$-orbits of subspaces as in
  Proposition~$\ref{p:reduce}$, let  $U\in \cU$, $U'\in\cU'$,
  and for $h\in G^\bX$ let $E(h) = U^\perp\cap (U'h)^\perp$. Let
\[
\cX \coloneq\{ U'h\mid h\in G^\bX,\ \mbox{and}\ V=E(h)\oPerp (U\oplus (U'h))\},
\quad\mbox{and}\quad T\coloneq\{ h\in G^\bX\mid U'h \in\cX\}.
\]
Then $|T|/|G^\bX|=\rho(\bX,V,\cU, \cU')$.
\end{lemma}

\begin{proof}
  Note that $U'h=U'h'$ if and only if $h'h^{-1}\in G_{U'}^\bX$. Thus each
  $U'h\in\cX$ occurs for exactly $|G_{U'}^\bX|$ distinct elements $h$ of $G^\bX$.
  Further, since $\cU'$ is a $G^\bX$-orbit, $|G_{U'}^\bX|=|G^\bX|/|\cU'|$. It follows
  that $|T|=|\cX|\cdot |G^\bX|/|\cU'|$, and hence 
  \[
  \frac{|T|}{|G^\bX|} = \frac{|\cX|}{|\cU'|} = \rho(\bX,V,\cU, \cU'),
  \] 
  where the last equality follows from Corollary~\ref{c:dllo1}.
\end{proof}

For the  algorithm in  \cite{DLLO13}, the  $G^\bX$-orbits $\cU,  \cU'$ are
identical,  and  we  simply draw  together  Theorem~\ref{t:main1}  and
Lemma~\ref{l:dllo3} for the case where  $e=e'$ and $\cU=\cU'$. We note
that the condition $V=E(h)\oPerp (U\oplus (U'h))$ in Lemma~\ref{l:dllo3}
implies   that  $\dim(E(h))=d-e-e'$. The following lemma proves
\cite{DLLO13}*{Lemma~5.8}  for all the classical groups as described
in Section~\ref{s:notation}(b). The bound $0.05$ below arises since the
smallest value of $1-c/|\F|$ for $c$ in Table~\ref{tab1}
arises for $c=2.85$ and $q=3$ by Theorem~\ref{T:orthobound}.

\begin{lemma}\label{l:ddlo4}
  Suppose that the hypotheses in {\rm Section~\ref{s:notation}(a,b,c)} hold
  and $\cU=\cU'$. Let $U\in\cU$ and 
\[
T = \{ h\in G^\bX \mid  V=E(h)\oPerp (U\oplus (Uh))\ \mbox{where}\ E(h) = U^\perp\cap (Uh)^\perp\}.
\]
Then  $|T|/|G^\bX|=\rho(\bX,V,\cU, \cU)$,
In particular, for all $\bX$ and $\cU$ we have $|T|/|G^\bX|>0.05$.
\end{lemma}

The relation between $|T|/|G^\bX|$ and $\rho(\bX,V,\cU, \cU)$
in Lemma~\ref{l:ddlo4} was not appreciated in~\cite{DLLO13} and hence
the authors did not foresee the difficult problem of finding a lower bound for
$\rho(\bX,V,\cU, \cU)$ in their proof
of the classical case in~\cite{DLLO13}*{Lemma~5.8}.

\begin{remark}
  (a) If we refine the bound in Lemma~\ref{l:ddlo4} to depend on $\bX$ and $|\F|$,
  we have $|T|/|G^\bX|\ge 1-c/|\F|$
  by Theorem~\ref{t:main1}  where $c$ is given in Table~\ref{tab1}.
  Thus very few random selections are needed in
  the algorithm in~\cite{DLLO13} for large~$q$.

  (b) Recall that our 
  methods are not strong enough to give a useful lower bound for orthogonal 
  groups over a field of order $q=2$, but we have bounds for all other cases. 
  The omission of this case is not an issue in correcting the proof of 
  \cite{DLLO13}*{Lemma~5.8} since in that paper the analysis is given for fields 
  of even size $q>4$, see \cite{DLLO13}*{Theorems 1.2 and 1.3, and Remark 1.5}. 
  In relation to the comment in \cite{DLLO13}*{Remark 1.5} about the
  restriction to $q>4$ being needed because it relies on results in
  \cite{PSY2015}, we note that the cases $q=3, 4$ are also covered in
  \cite{PSY2015}*{Theorem 2}, and moreover \cite{PSY2015}*{Theorems 5
    and 6} are valid for all field sizes, the only exception being
  orthogonal groups with $q=2$. Thus the only exclusion in the
  analyses in both \cite{DLLO13} and \cite{PSY2015} is for orthogonal
  groups with $q=2$. The results for very small fields given in
  \cite{PSY2015} were not in the preprint available to the authors of
  \cite{DLLO13} at the time of its publication. 
\end{remark}

\subsection{Stingray elements}\label{subsec:stingelts}


Let $G=G^\bX$ be a group as in Section~\ref{s:notation}(b).
In this subsection we study elements $g\in G$ for which the image
$\im(g-1)$ is non-degenerate. For such elements $g$, the $G$-conjugacy
class $\cC=g^G$ corresponds to the $G$-orbit $\cU=\{\im(g'-1)\mid
g'\in \cC\}$ of non-degenerate subspaces.  Moreover, for two such
$G$-conjugacy classes $\cC$ and $\cC'$, and corresponding $G$-orbits
$\cU$ and $\cU'$, we establish in Lemma~\ref{l:stingequalspaces} that
the proportion of $\bX$-duos in $\cU\times\cU'$ is equal to the
proportion of certain kinds of pairs in $\cC\times\cC'$ which we call
stingray-duos, see Definition~\ref{def:stingray}(d). This connection
is crucial for our algorithmic applications. We make the following definitions.

\begin{definition}\label{def:stingray}
  Assume that the hypotheses of Section~\ref{s:notation}(a,b) hold with $G=G^\bX$, and
  that $d\ge 2$. 
  \begin{enumerate}[{\rm (a)}]
  \item For $g\in G$ let $F_g\coloneq \ker(g-1)$ be the $1$-eigenspace
    (fixed-point space) of~$g$ in $V$ and let $U_g\coloneq\textup{im}(g-1)$.
    Note that $U_g$ and $F_g$ are $\langle g\rangle$-invariant,
      and $\dim(U_g)+\dim(F_g)=\dim(V)$.
    
  \item For a positive integer $e\le d$, an element $g\in G$ is called
    an \emph{$e$-stingray   element} if $\dim(U_g)=e$ and~$g$
    acts irreducibly on $U_g$.
      
  \item  For positive integers $e, e'$ such that $e, e'$ and
    $e + e' \le d$, we  say  that $(g,g')\in G\times G$  is  an
    \emph{$(e,e')$-stingray pair},  if $g$ is an $e$-stingray element and
    $g'$ is an $e'$-stingray element. 

  \item  An  $(e,e')$-stingray pair  $(g,g')$  is called an
  \emph{$(e,e')$-stingray duo}, if $e\ge e'$ and $(U_g,U_{g'})$ is an $\bX$-duo
  as in Section~\ref{s:notation}(d).
\end{enumerate}
\end{definition}

\begin{lemma}\label{lem:unique2}
  Let $G=G^\bX$ as in {\rm \,Section~\ref{s:notation}(b)}
  and let $g\in G$ be an $e$-stingray element with $U_g, F_g$ as in
  Definition~$\ref{def:stingray}$. Then
  \begin{enumerate}[{\rm (a)}]
  \item $0\ne v^g-v\in U_g$, for all $v\in V\backslash F_g$;
  \item if $Z$ is a $\langle g\rangle$-invariant
    submodule of $V$ then either $g$ acts trivially on $Z$ and $Z\le F_g$, or
     $g$ acts non-trivially on $Z$, $U_g\le Z$ and
    the restriction $g_{\mid Z}$ of $g$ to $Z$ is an $e$-stingray
    element of $\GL(Z)$;
  \item  $U_g$ is the unique  $\langle g\rangle$-invariant submodule
    of $V$ on which $g$ acts  non-trivially and irreducibly.
  \end{enumerate}
\end{lemma}

\begin{proof}
  (a)~Let $v\in V\backslash F_g$. Then there exist
  $u\in U_g$ and $f\in F_g$ such that $v=u+f$ and $u\neq 0$. 
    The result follows since $v^g-v = u^g-u \in U_g$, and $v^g-v = 0$
    if and only if $v \in F_g$.

  (b)~Now  suppose that  $Z$ is a $\langle  g\rangle$-invariant
  submodule of  $V$. If $g$ acts trivially on $Z$ then clearly
  $Z \le F_g.$ Suppose  $g$ acts non-trivially on $Z$. Then there
  exists some $v\in Z$ with $v^g\ne v$ and so $v^g-v\in U_g$ is nonzero.
  Since $Z$ is $\langle g\rangle$-invariant, and $\langle g\rangle$
  is irreducible on $U_g$, it follows that $U_g\le Z$.
  Further $Z=U_g\oplus (F_g\cap Z)$ and $g_{\mid Z}$ is an $e$-stingray element
  of $\GL(Z)$ as claimed.

  (c)~Let $W$ be a $\langle g\rangle$-invariant submodule
    on which $g$ acts irreducibly and non-trivially. 
	Since $g$ acts  non-trivially on $W$ it follows from part (b) that 
	$U_g\le W$. Then, since $g$ acts irreducibly on $W$, this implies that 
	$U_g=W$.      
\end{proof}

\begin{lemma}\label{l:prop1}
  Let $G=G^\bX$ be a  group as in {\rm \,Section~\ref{s:notation}(b)} of type $\bX$
  and let $g\in G$ be an $e$-stingray element with $U_g, F_g$ as in
  Definition~$\ref{def:stingray}$. Then 
  \begin{enumerate}[{\rm (a)}]
  \item $V = U_g\oPerp F_g$, and in particular, $U_g$ and $F_g$ are
    non-degenerate and $U_g^\perp = F_g$;
  \item the parity of $e$ is as  given in
    Table~$\ref{tab:one}$, where if
    $\bX=\bO^\eps$, then $U_g$ has  minus type.
    
  \end{enumerate}
\end{lemma}

\begin{proof}
  (a)~This observation dates back at least to~\cite{Wall}*{Corollary p.\,6}.
  Let $\beta\colon V\times V\to\F$ denote
  the   non-degenerate sesquilinear form preserved by $G$.
  Then,
  for $v \in V$ and $f\in F_g$, we have
  $\beta(v, f) = \beta(v^g , f^g ) = \beta(v^g, f )$, so $\beta(v-v^g, f) = 0$.
  Hence $U_g=V(1-g)\le F_g^\perp$. However, $\dim(V)=\dim(U_g)+\dim(F_g)$,
  so $\dim(U_g)=\dim(F_g^\perp)$. Hence $U_g=F_g^\perp$ and so $U_g^\perp=F_g$.
  Since $U_g\cap F_g$ is $\langle g\rangle$-invariant and since $g$ acts irreducibly on $U_g$, it follows that $U_g\cap F_g$ is trivial. Thus
  $U_g\cap U_g^\perp = 0$ since $U_g^\perp=F_g$, and hence $U_g$
  is non-degenerate. Also $F_g=U_g^\perp$ is non-degenerate and $V=U_g\oPerp F_g$.

  (b)~By the definition of an $e$-stingray element, the characteristic 
  polynomial~$c_g(t)$ of $g$ satisfies $c_g(t)=(t-1)^{d-e}c_{h}(t)$,
  where $h=g_{\mid_{U_g}}$ denotes the restriction of $g$ to~$U_g$. Moreover, 
  $U_g$ is an irreducible $\langle h\rangle$-module, so 
  $c_{h}(t)$ is a monic irreducible polynomial over $\F$ of degree $e$, 
  and in particular   $c_h(0)\ne0$.  For any polynomial $f(t)$ over 
  $\F$ with $f(0)\ne0$, let $f^{\textup{rev}}(t)=f(0)^{-1}t^{\deg(f)}
  f(t^{-1})$, the  \emph{reverse polynomial} of $f(t)$. Also, if 
  $\bX=\bU$ and $J\in\GL(V)$, let $J^\phi$ denote 
  the matrix obtained from $J$ by applying $q$th powers to each entry.
  Now since $g\in G^\bX$, we have
  $gJg^\phi=J$ or $gJg^{\rm T}=J$, according as $\bX=\bU$ or 
  $\bX\in\{\bSp,\bO^\varepsilon\}$ respectively, where in both 
  cases $J$ is an (invertible) Gram matrix.  
  Therefore $J^{-1}gJ=g^{-\phi}$ or $J^{-1}gJ=g^{\rm -T}$, 
  and so~$c_g(t)=c_g^{\textup{rev}}(t)^\phi$ or 
  $c_g(t)=c_g^{\textup{rev}}(t)$, respectively.
  It follows that~$c_{h}(t)=c_{h}^{\textup{rev}}(t)^\phi$ or 
  $c_{h}(t)=c_{h}^{\textup{rev}}(t)$, and (since $c_h(t)$ is a 
  monic irreducible) that   $e$ is odd, or $e$ is even, respectively, 
  by~\cite{FNP}*{Lemma~1.3.11(b) and Lemma~1.3.15(c)}.  
  Thus we obtain the restrictions on the parity of $e$ in 
  Table~\ref{tab:one}.
  Finally, if $\bX=\bO^\varepsilon$, then $g_{\mid U_g}$ irreducible implies that 
  the type of $U_g$ is~\emph{minus}, see~\cite{B}*{pp.\,187--188} for example.
\end{proof}

  We shall be studying $(e,e')$-stingray pairs in a group $G=G^\bX$ 
  as in {\rm \,Section~\ref{s:notation}(b)}. In the case where $\bX=\bO^\eps$, 
  we assume that $e, e'$ are both even so in particular $d-e\ge2$. Thus the 
  assumptions in the following lemma will always hold.

\begin{lemma}[\cite{KL}*{Lemma~4.1.1(iv,v)}]\label{l:red}
  Let $G=G^\bX$ be a  group as in {\rm \,Section~\ref{s:notation}(b)} of type
  $\bX$, and let $U$ be a non-degenerate proper subspace
  of~$V$. Moreover if $\bX=\bO^\eps$ assume also that $\dim(V) - \dim(U)\ge 2$.
  Then the group $G^U$ induced (via restriction) on $U$ by the
  setwise stabiliser $G_U$ is the
  full isometry group $\GX(U)$.
\end{lemma}

\begin{remark}
   Lemma~\ref{l:red} follows from~\cite{KL}*{Lemma~4.1.1(iv,v)}. However in 
   applying this result 
  we note that the statement of~\cite{KL}*{Lemma~4.1.1} involves the hidden
  assumption $\dim(U)\ge\dim(V)/2$ (see~\cite{KL}*{p.\,83, Definition}).
  Thus the conclusion that $G^U$ induces $\GX(U)$ follows from part (iv) 
  of~\cite{KL}*{Lemma~4.1.1} if $\dim(U)\le \dim(V)/2$, and from part (v)  
  of~\cite{KL}*{Lemma~4.1.1} if $\dim(U)>\dim(V)/2$, 
  noting that, in the latter case, our assumption that $\dim(U)\le \dim(V)-2$ 
  when $\bX=\bO^\eps$  avoids the exception
  in~\cite{KL}*{Lemma~4.1.1(v)}.
\end{remark}


\begin{table}
  \caption{The parity of $e$ for different $\bX$.}
\begin{tabular}{rlllll}
\toprule
$\bX$ &&  $\bU$ &$\bSp$ & $\bO^\pm$& $\bO^\circ$\\
\midrule
$\OmX(V)$ &&  $\SU_d(q)$ &$\Sp_d(q)$ & $\Omega^\pm_d(q)$& $\Omega^\circ_d(q)$\\
parity of $e$&& odd&even&even&even\\
\bottomrule
\end{tabular}
\label{tab:one}
\end{table}


\subsubsection{Stingray elements and subspaces}\label{subsub:stingray}

Let $G=G^\bX$ be a group as in Section~\ref{s:notation}(b), let $\cC$ be a $G$-conjugacy class of $e$-stingray elements as in Definition~\ref{def:stingray}(b), and let $g\in G$.
Then by~Lemmas~\ref{lem:unique2}(c) and \ref{l:prop1}, $U_g=\im(g-1)$ is the unique 
$\langle g\rangle$-invariant subspace of $V$ on which $g$ acts non-trivially and irreducibly, $U_g$ is non-degenerate, and $\dim(U_g)=e$ with the parity of $e$ as in Table~\ref{tab:one}. Thus $\cU=\{U_g\mid g\in \cC\}$ is a $G$-orbit of
non-degenerate subspaces, and so $\cU\subseteq\Binom{V}{e}^\bX$ as described
in Section~\ref{s:notation}(c). 
Clearly $G$ acts transitively  via conjugation
on $\cC=g^G$, and the stabiliser of $g$ is $C_G(g)$.
Further, $C_G(g)$ leaves both $U_g$ and $F_g=U_g^\perp$ invariant
and so $C_G(g)\le G_{U_g}\le G$.
We next relate $|\cC|$ and $|\cU|$.


\begin{lemma}\label{lem:SizeOfC}
Let $\cC, \cU$ be as above, and let $g\in\cC$ and $U=U_g$. Then 
  \begin{equation}\label{e:CU}
    |\cC| = |\cU|\cdot |G_{U}: C_G(g)|,
  \end{equation}
  and there are precisely $|\cC|/|\cU|$ elements $g'\in\cC$ such that $U_{g'}=U$.
\end{lemma}

\begin{proof} It follows from $C_G(g)\le G_{U}\le G$ that
\[
  |\cC| = |G:C_G(g)| = |G:G_{U}| \cdot |G_{U}: C_G(g)|
    = |\cU| \cdot |G_{U}: C_G(g)|.
\]
  Therefore the action of $G$ on $\cC$ preserves the partition of $\cC$ into
  classes $\cC(U')$ for $U'\in\cU$ where $\cC(U'):=\{ g'\in \cC \mid\ U_{g'}=U'\}$,
  and the number of conjugates in each of these classes is $|G_{U}: C_G(g)|$. 
\end{proof}

 \subsubsection{Stingray duos}
 Suppose that the hypotheses of Section~\ref{s:notation}(a,b) hold, let $G=G^\bX$, and
 let ${\mathcal C}$ be a $G$-conjugacy class of $e$-stingray elements
 and ${\mathcal C}'$ a $G$-conjugacy class of $e'$-stingray elements
 such that $e\ge e'$ and $e+e'\le d$. Thus each $(g,g')\in {\mathcal C}\times
 {\mathcal C}'$ is an $(e, e')$-stingray pair, as in
 Definition~\ref{def:stingray}.  As defined there, we will say that
 $(g,g')$ is a \emph{stingray duo} if $U_g\cap U_{g'}=0$ and
 $U_g\oplus U_{g'}$ is non-degenerate.
 Let $\cU=\{U_g\mid g\in\cC\}$ and $\cU'=\{U_{g'}\mid g'\in\cC'\}$, so that, as noted in the previous subsection, $\cU$ is a $G$-orbit
 contained in $\Binom{V}{e}^\bX$, and $\cU'$ 
 is a $G$-orbit contained in $\Binom{V}{e'}^\bX$. We denote the proportion of stingray duos in $ {\mathcal C}\times {\mathcal C}'$ by  
  \begin{align}\label{e:rhouseful}
  \rhosting &=  \frac{|\{(g,g') \in \cC\times \cC'\mid
    \textup{$(e,e')$-stingray\ duo}\}|}{|\cC\times \cC'|}.
  \end{align}

 \begin{lemma}\label{l:stingequalspaces}
 Let $\cC, \cC',\cU, \cU'$ be as above. Then $\rhosting = \rho(\bX,V,\cU,\cU')$.
\end{lemma}

 \begin{proof}
 It follows from Lemma~\ref{lem:SizeOfC} that each $U\in\cU$ is equal to $U_g$ for exactly $|\cC|/|\cU|$ elements $g\in\cC$. Similarly, each $U'\in\cU'$ is equal to $U_{g'}$ for exactly $|\cC'|/|\cU'|$ elements $g'\in\cC'$. By Definition~\ref{def:stingray}(d), a pair $(g,g')\in\cC\times \cC'$ is an $(e,e')$-stingray duo if and only if $(U_g, U_{g'})$ is an $\bX$-duo in $\cU\times \cU'$. Hence the number of  $(e,e')$-stingray duos in 
 $\cC\times \cC'$ is equal to $(|\cC|/|\cU|)\cdot (|\cC'|/|\cU'|)\cdot Y$,
 where $Y$ is the number of  $\bX$-duos in $\cU\times \cU'$. By \eqref{d:rho}, 
 $Y= |\cU|\cdot |\cU'|\cdot  \rho(\bX,V,\cU,\cU')$, and hence the number of $(e,e')$-stingray duos in 
 $\cC\times \cC'$ is $|\cC|\cdot |\cC'|\cdot  \rho(\bX,V,\cU,\cU')$. It follows that 
 $\rhosting = \rho(\bX,V,\cU,\cU')$.
\end{proof}

\section{Counting non-degenerate subspaces}\label{sec:nondegsub}
Let $V$ be a $d$-dimensional classical space over a finite field $\F$.
An $e$-dimensional subspace of $V$ is called an
\emph{$e$-subspace} and the set of all of non-degenerate $e$-subspaces
is denoted by $\Binom{V}{e}$.
In the orthogonal case, $V$ has type $\eps\in\{-,\circ,+\}$ and the set of
non-degenerate $e$-subspaces of subspace type $\tau$ is denoted
$\Binom{V}{e}^\eps_\tau$.
For our application, $e$ is even so $\tau\in\{-,+\}$, subspace type equals intrinsic type, and~$d$ may be odd or even.
We assume that $q$ is odd if $d$ is odd. Hence $V=U\oPerp U^\perp$ holds for
$U\in\Binom{V}{e}^\eps_\tau$. Furthermore, if $U$ has type $\tau\in\{-,+\}$,
then it follows from~\cite{KL}*{Proposition~2.5.11(ii)} that $U^\perp$
has type~$\eps\tau$ (even when $d$ is odd).

The order of a classical group can be expressed as a power of $q$ times
a rational function in $q$ which approaches 1 as $q\to\infty$.
Such a rational function is
\begin{equation}\label{E:omega}
\omega(d,q)\coloneq\prod_{i=1}^{d}(1-q^{-i})
    \qquad\textup{for $d\ge0$ where $\omega(0,q)=1$.}
\end{equation}
    The orders of the isometry groups
      $\GU_d(q), \Sp_d(q), \GO_d^\eps(q)$
      given in~\cite[Table~2.1.C,\,p.19]{KL} are rewriten below
      in terms of the `dominant  power of $q$'.
In~\eqref{E:SpOord} below we identify the symbols $-,\circ,+$ with the
numbers $-1,0,1$, respectively. Then
\begin{align}
  |\GL_d(q)|&=q^{d^2}\omega(d,q), \qquad&&|\GU_d(q)|=q^{d^2}\omega(d,-q), \label{E:Uord}\\
  |\Sp_d(q)|&=q^{\binom{d+1}{2}}\omega(\frac{d}{2},q^2), 
  &&|\GO_d^\eps(q)|=2q^{\binom{d}{2}}\frac{\omega(\lfloor\frac{d}{2}\rfloor,q^2)}
                               {1+\eps q^{-\lfloor\frac{d}{2}\rfloor}}.\label{E:SpOord}
\end{align}

In the unitary case of Proposition~\ref{prop:pA}, we assume that $1\le e\le d-1$, and in the symplectic
and orthogonal cases we assume that $e$ is even and $2\le e\le d-2$.


\begin{proposition}\label{prop:pA}  
  Let $V$ be a non-degenerate $d$-dimensional classical space. Then
  \begin{enumerate}[{\rm (a)}]
  \item 
    The number of non-degenerate $e$-subspaces of a unitary
    space $V=(\F_{q^2})^d$ is
    \[
    \left|\Binom{V}{e}^{\bf U}\right|
    =\frac{q^{2e(d-e)}\omega(d,-q)}{\omega(e,-q)\omega(d-e,-q)}.
    \]
  \item 
    The number of non-degenerate $e$-subspaces of a symplectic
    space $V=(\F_q)^d$ is
    \[
    \left|\Binom{V}{e}^{\bf Sp}\right|
    =\frac{q^{e(d-e)}\omega(\frac{d}{2},q^2)}
          {\omega(\frac{e}{2},q^2)\omega(\frac{d-e}{2},q^2)}.
    \]
  \item
    Let $V=(\F_q)^d$ be an orthogonal space of type $\eps$. For $e$ even,
    the number of non-degenerate $e$-subspaces of type $\tau\in\{-,+\}$ is
  \begin{align*}
    \left|\Binom{V}{e}^\eps_{\tau}\right|&=
    \displaystyle\frac{q^{e(d-e)}(1+\tau q^{-\frac{e}{2}})
      (1+\eps\tau q^{-\lfloor\frac{d}{2}\rfloor+\frac{e}{2}})}
                      {2(1+\eps q^{-\lfloor\frac{d}{2}\rfloor})}\cdot
                      \frac{\omega(\lfloor\frac{d}{2}\rfloor,q^2)}
          {\omega(\frac{e}{2},q^2)\omega(\lfloor\frac{d}{2}\rfloor-\frac{e}{2},q^2)}.
  \end{align*}
  \end{enumerate}
\end{proposition}

\begin{proof}
  The stabiliser of a non-degenerate subspace $U$ of $V$
    equals the stabiliser of the decomposition $V=U\oPerp U^\perp$
    by~\cite[Lemma~2.1.5(v)]{KL}, and the shape of these stabilisers is 
    given in \cite[Table 4.1.A]{KL}. We use the formulas~\eqref{E:Uord}
    and~\eqref{E:SpOord}.
  
  (a)~It follows from Witt's Theorem, the orbit-stabiliser lemma
  and~\eqref{E:Uord}, that the number of non-degenerate
  $e$-subspaces of $V$ is
  \[
  \frac{|\GU_d(q)|}{|\GU_e(q)||\GU_{d-e}(q)|}
  =\frac{q^{d^2}\omega(d,-q)}{q^{e^2}\omega(e,-q)q^{(d-e)^2}\omega(d-e,-q)}
  =\frac{q^{2e(d-e)}\omega(d,-q)}{\omega(e,-q)\omega(d-e,-q)}.
  \]

  (b)~This proof is similar to part~(a).

  (c)~Suppose that $U$ is a non-degenerate $e$-subspace of $V$ of type $\tau$.
  By the preamble to this proposition, $V=U\oPerp U^\perp$ and $U^\perp$ has
  type $\eps\tau$. Hence the stabiliser of $U$ in $\GO_d^\eps(q)$
  is $\GO_e^\tau(q)\times\GO_{d-e}^{\eps\tau}(q)$ by Witt's Theorem.
  It follows from the orbit-stabiliser lemma and~\eqref{E:SpOord} that
  $\left|\Binom{V}{e}^\eps_{\tau}\right|$ equals
  \[
  \frac{|\GO_d^\eps(q)|}{|\GO_e^\tau(q)\times\GO_{d-e}^{\eps\tau}(q)|}=
      \displaystyle\frac{q^{\binom{d}{2}-\binom{e}{2}-\binom{d-e}{2}}
      (1{+}\tau q^{-\frac{e}{2}})
      (1{+}\eps\tau q^{-\lfloor\frac{d}{2}\rfloor+\frac{e}{2}})}
                      {2(1+\eps q^{-\lfloor\frac{d}{2}\rfloor})}\cdot
                      \frac{\omega(\lfloor\frac{d}{2}\rfloor,q^2)}
          {\omega(\frac{e}{2},q^2)\omega(\lfloor\frac{d}{2}\rfloor{-}\frac{e}{2},q^2)}.
   \]
   However, $\binom{d}{2}-\binom{e}{2}-\binom{d-e}{2}=e(d-e)$ and
   the result follows.
\end{proof}

\section{Bounding rational functions in \texorpdfstring{$q$}{}}\label{sec:nt}

This section derives bounds that are used to estimate functions in
our main theorems.
The following infinite products provide useful limiting bounds:
\[
\omega(\infty,q)=\prod_{i=1}^\infty(1-q^{-i})
\quad\textup{and}\quad
\omega(\infty,-q)=\prod_{j=1}^\infty(1-(-q)^{-j})=\prod_{i=1}^\infty(1+q^{-(2i-1)})(1-q^{-2i})
.
\]

\begin{lemma}\label{lem:omegabounds}
Suppose $q>1$. Using the notation~\eqref{E:omega} we have:
\begin{align*}
&1-q^{-1}-q^{-2}+q^{-5}<\omega(\infty,q)<\cdots<\omega(2,q)<\omega(1,q)<\omega(0,q)=1\quad\textup{and,}\\
&1=\omega(0,-q)<\omega(2,-q)<\cdots<\omega(\infty,-q)<\cdots<\omega(3,-q)<\omega(1,-q)=1+q^{-1}.
\end{align*}
Moreover,
\[
  \lim_{n\to\infty}\omega(n,q)=\omega(\infty,q)
  \quad\textup{and}\quad
  \lim_{m\to\infty}\omega(2m,-q)=\lim_{m\to\infty}\omega(2m+1,-q)=\omega(\infty,-q).
\]
\end{lemma}

\begin{proof}
  For bounds such as $1-q^{-1}-q^{-2}<\omega(\infty,q)$ or
  $1-q^{-1}-q^{-2}+q^{-5}<\omega(\infty,q)$, see~\cite{NP}*{Lemma~3.5}.
  Since
  $\omega(n,q)=\omega(n-1,q)(1-q^{-n})<\omega(n-1,q)$, the first inequalities
  follow. Clearly $\omega(n,q)\to\omega(\infty,q)$ as $n\to\infty$.
  The second inequalities use
  \begin{equation}\label{E:bb}
    (1-q^{-2i})(1+q^{-(2i+1)})<1<(1+q^{-(2i-1)})(1-q^{-2i}).
  \end{equation}
  Not only does this imply that the infinite product $\omega(\infty,-q)$
  converges to a positive limit for all $q>1$, it shows that
  $\lim_{m\to\infty}\omega(2m,-q)=\lim_{m\to\infty}\omega(2m+1,-q)=\omega(\infty,-q)$.
\end{proof}

By virtue of the symmetry $\binom{n}{n-k}_q<\binom{n}{k}_q$,
we assume in the following lemma that $k\le\lfloor\frac{n}{2}\rfloor$.
Similarly, since
$\binomqq{n}{n-k}=\binomqq{n}{k}$ we assume $k\le\lfloor\frac{n}{2}\rfloor$
whether $k$ is even or odd.

\begin{lemma}\label{lem:binombounds}
  Suppose $q>1$ and $k\in\{0,1,\dots,n\}$. Using~\eqref{E:omega}, we define
  \begin{equation}\label{E:binom}
    \binomq{n}{k}\coloneq\frac{\omega(n,q)}{\omega(k,q)\omega(n-k,q)}
    \quad\textup{and,}\quad
    \binomqq{n}{k}\coloneq\frac{\omega(n,-q)}{\omega(k,-q)\omega(n-k,-q)}.
  \end{equation}
  Then
\[
1=\binomq{n}{0}<\binomq{n}{1}<\binomq{n}{2}<\binomq{n}{3}<\cdots<\binomq{n}{\lfloor \frac{n}{2}\rfloor}<\frac{1}{\omega(\infty,q)}<\frac{1}{1-q^{-1}-q^{-2}}.
\]
Further, if $i\coloneq\lfloor\frac{1}{2}\lfloor\frac{n}{2}\rfloor-\frac{1}{2}\rfloor$
and $j\coloneq\lfloor\frac{1}{2}\lfloor\frac{n}{2}\rfloor\rfloor$, then
$2i+1\le\lfloor\frac{n}{2}\rfloor$, $2j\le\lfloor\frac{n}{2}\rfloor$ and
\begin{align*}
  \frac{1-(-q)^{-n}}{1+q^{-1}}&=\binomqq{n}{1}<\binomqq{n}{3}<\cdots
  <\binomqq{n}{2i+1}
  <\frac{1}{\omega(\infty,-q)}\\
  \mbox{and}\quad\frac{1}{\omega(\infty,-q)}&<\binomqq{n}{2j}<\cdots<\binomqq{n}{2}<\binomqq{n}{0}=1.
\end{align*}
\end{lemma}

\begin{proof}
  Suppose that $1\le k\le\lfloor\frac{n}{2}\rfloor$.
  Cancelling  $\omega(n-k,q)$ and $\omega(n-k,-q)$ gives
  \begin{equation}\label{E:bat}
    \binomq{n}{k}
    =\frac{\prod_{i=1}^{k}(1-q^{-(n-k+i)})}{\omega(k,q)}
    \quad\mbox{and}\quad
    \binomqq{n}{k}
    =\frac{\prod_{i=1}^{k}(1-(-q)^{-(n-k+i)})}{\omega(k,-q)}.
  \end{equation}
  
  The inequality
  $\binom{n}{k-1}_q<\binom{n}{k}_q$ holds as $\omega(k-1,q)>0$
  and $1-q^{-k}<1-q^{-(n-k+1)}$. It follows from~\eqref{E:bat}
  and Lemma~\ref{lem:omegabounds}
  that $\binomq{n}{\lfloor\frac{n}{2}\rfloor}<\frac{1}{\omega(\infty,q)}$.
  Hence the first chain of inequalities is true.
  Similar reasoning using
  $\omega(k-2,-q)>0$ and~\eqref{E:bb} establishes the remaining inequalities
  with the exception of
  $\binomqq{n}{2i+1} <\frac{1}{\omega(\infty,-q)}<\binomqq{n}{2j}$.
  This follows from~\eqref{E:bat}, the definitions of $i$ and $j$ in
  terms of $n$, and
  $\lim_{n\to\infty}\binomqq{n}{2i+1} =\lim_{n\to\infty}\binomqq{n}{2j}=\frac{1}{\omega(\infty,-q)}$.
\end{proof}

The bounds in Lemmas~\ref{lem:omegabounds} and~\ref{lem:binombounds} have
largest error for small $q$. For example,
\[
\omega(\infty,2)=0.288\cdots, \omega(\infty,-2)=1.210\cdots, 
\omega(\infty,3)=0.560\cdots, \omega(\infty,-3)=1.217\cdots.
\]

In Lemmas~\ref{lem:omMinus} and~\ref{lem:int}, the functions
$\textup{exp}(x)$, $\log(x)$ have
the natural base $e=2.718\cdots$.

\begin{lemma}\label{lem:omMinus}
  Suppose $n\ge0$, $q>1$ and $\omega_{-}(n,q)=\prod_{i=1}^n(1+q^{-i})$. Then
  \[
  \frac{1-q^{-(n+1)}}{1-q^{-1}} \le \omega_{-}(n,q)
  \le \textup{exp}\left(\frac{q^{-1}(1-q^{-n})}{1-q^{-1}}\right)
  \quad\textup{and}\quad
  1-q^{-1}\le \omega_{-}(n,-q) \le 1.
  \]
\end{lemma}

\begin{proof}
  The inequalities follow from 
  $1+\sum_{i=1}^n a_i\le\prod_{i=1}^n(1+a_i)\le \textup{exp}(\sum_{i=1}^n a_i)$,
  where each $a_i$ is positive, and
  $(1-q^{-(2i-1)})(1+q^{-2i})<1<(1+q^{-2i})(1-q^{-(2i+1)})$.
\end{proof}

\begin{lemma}\label{lem:int}
  Suppose $a,b\in\Z$ and $q\in\R$ where $1\le a\le b$ and $q>1$. Then
  \begin{enumerate}[{\rm (a)}]
    \item 
      $\displaystyle 1-\frac{q^{-a}}{\log(q)}
      < 1-\frac{q^{-a}-q^{-b}}{\log(q)}
      \le \frac{\omega(b,q)}{\omega(a,q)}\le 1$, and
    \item 
      $\displaystyle 1 - q^{-2}\le 1 - q^{-a-1}\le\frac{\omega(b,-q)}{\omega(a,-q)}\le 1 + q^{-a-1}\le 1 + q^{-2}$.
  \end{enumerate}
\end{lemma}

\begin{proof}
  (a) The upper bound of 1 is immediate, and the other bounds follows from
  \begin{align*}
    \frac{\omega(b,q)}{\omega(a,q)} & = 
    \prod_{i=a+1}^b(1-q^{-i})  
    \ge 1 - \sum_{i=a+1}^b q^{-i}
     \ge 1 - \int_{a}^{b} q^{-x} dx 
     = 1 - \frac{q^{-a}-q^{-b}}{\log(q)}
        > 1 - \frac{q^{-a}}{\log(q)}.
  \end{align*}

  (b) Define $\zeta_n\coloneq (1-(-q)^{-(n-1)})(1-(-q)^{-n})$.
  First consider the upper bounds. We use the fact that
  $\zeta_n < 1$ for $n$ odd by \cite{Forms}*{Lemma~4.2(a)}. Write
 \begin{align*}
   \frac{\omega(b,-q)}{\omega(a,-q)} &= 
   \prod_{i=a+1}^b(1-(-q)^{-i})= A\cdot \left(\prod \zeta_n\right) \cdot B,
 \end{align*}
 where the product ranges over all odd $n$ with $a+2 \le n \le b$,
 and $A = 1$ if $a$ is odd and $A= 1+q^{-a-1}$ if $a$ is even,
 and $B = 1$ if $b$ is odd and $B=1-q^{-b}$ if $b$ is even.
 As $\zeta_n <1$ for $n$ odd, we have
 $A\cdot \left(\prod \zeta_n\right) \cdot B \le  A \cdot B
 \le A \le 1+q^{-a-1}\le 1 + q^{-2}$. 

 For the the lower bounds we use the fact that
 $\zeta_n > 1$ for $n$ even by \cite{Forms}*{Lemma~4.2(a)}:
  \begin{align*}
\frac{\omega(b,-q)}{\omega(a,-q)} &= 
\prod_{i=a+1}^b(1-(-q)^{-i})
= A\cdot \left(\prod \zeta_n\right) \cdot B,
 \end{align*}
 where the product ranges over all even $n$ with $a+2 \le n \le b$
 and $A = 1$ if $a$ is even and $A= 1-q^{-a-1}$ if $a$ is odd
 and $B = 1$ if $b$ is even and $B=1+q^{-b}$ if $b$ is odd.
 As $\zeta_n >1$ for $n$ even, we have
 $A\cdot (\prod \zeta_n)\cdot B \ge  A\cdot B \ge A \ge 1-q^{-a-1}\ge 1 - q^{-2}.$
\end{proof}

\begin{lemma}\label{lem:omegabound}
  Let $m, a, b$ be positive integers with $m > a + b$. Then
  \[
  \frac{\omega(m-a,q) \omega(m-b,q)}{\omega(m-a-b,q) \omega(m,q)}
  > 1-\frac{1}{q\log(q)}. 
  \]
\end{lemma}

\begin{proof}
  Suppose first that $m=a+b+1$. It follows from Lemma~\ref{lem:omegabounds} that
  \begin{align*}
  \frac{\omega(m-a,q) \omega(m-b,q)}{\omega(m-a-b,q) \omega(m,q)}
  &= \frac{\omega(b+1,q)\omega(a+1,q) }{\omega(1,q) \omega(a+b+1,q)}\\
  &\ge \frac{\omega(b+1,q)}{\omega(1,q)} >
  \frac{ \omega(\infty,q)}{\omega(1,q)} > \frac{1-q^{-1}-q^{-2}+q^{-5}}{1-q^{-1}}.
  \end{align*}
  Moreover,  $ \frac{1-q^{-1}-q^{-2}+q^{-5}}{1-q^{-1}} \ge  1-\frac{1}{q\log(q)}$
  since  $q-1 \ge \log(q)$, so the lower bound holds.

  Finally, if $m > a+b+1$ then,
  using Lemma~\ref{lem:int}(a), we obtain
  \begin{align*}
    \frac{\omega(m-a,q)}{\omega(m-a-b,q)}
    \ge 1-\frac{1}{q^{m-a-b}\log(q)}
        \ge 1-\frac{1}{q^{2}\log(q)}.
  \end{align*}
  Moreover, again by Lemma~\ref{lem:int}(a) we obtain
  $\frac{\omega(m-b,q)}{\omega(m,q)} \ge 1$ and thus 
  \[ \frac{\omega(m-a,q) \omega(m-b,q)}{\omega(m-a-b,q) \omega(m,q)}
    \ge 1-\frac{1}{q^{2}\log(q)}
    > 1-\frac{1}{q\log(q)}.\qedhere
  \]
\end{proof}

\section{Transitioning from \texorpdfstring{$V$}{}
  to \texorpdfstring{$W$}{} and proving the bounds}\label{sec:proofs}

The goal of this section is to reduce from $V$ to the $(e+e')$-dimensional
subspace
$W$ in Proposition~\ref{p:reduce}, so we can apply
\cite{Forms}*{Theorem~1.1}. In Subsections~\ref{ss:Sp},~\ref{ss:U}
and \ref{ss:O} we shall find lower bounds for the ratios
\[
  \frac{|\Binom{V}{e+e'}^\bSp|\cdot |\Binom{W}{e}^\bSp|\cdot |\Binom{W}{e'}^\bSp|}
       {|\Binom{V}{e}^\bSp|\cdot |\Binom{V}{e'}^\bSp|}, \quad
  \frac{|\Binom{V}{e+e'}^\bU|\cdot |\Binom{W}{e}^\bU|\cdot |\Binom{W}{e'}^\bU|}
       {|\Binom{V}{e}^\bU|\cdot |\Binom{V}{e'}^\bU|}, \quad
  \frac{|\Binom{V}{e+e'}^\eps_\tau|\cdot |\Binom{W}{e}^\tau_\sigma|\cdot |\Binom{W}{e'}^\tau_{\sigma'}|}
       {|\Binom{V}{e}^\eps_\sigma|\cdot |\Binom{V}{e'}^\eps_{\sigma'}|}. 
\]
This allows us to apply Proposition~\ref{p:reduce}(a) in the symplectic
and unitary cases and Proposition~\ref{p:reduce}(b) in the orthogonal case.

We paraphrase \cite{Forms}*{Theorem~1.1} in
Theorems~\ref{T:Forms} and~\ref{T:orthobound} in a form that is useful for us.

\begin{theorem}[\cite{Forms}*{Theorem~1.1}]\label{T:Forms}
  Let $e, e'$ be positive integers with both even in part~{\rm (a)}.
  \begin{enumerate}[{\rm (a)}]
  \item Let $V=(\F_q)^{e+e'}$ be a non-degenerate symplectic space. Then the
    proportion of pairs $(U,U')$ of non-degenerate subspaces of dimensions
    $e$ and $e'$ respectively that satisfy $U\cap U'=0$,
    and hence $V=U\oplus U'$,
    is at least $1-\frac{5}{3q}>0$ for all $q\ge2$.
  \item Let $V=(\F_{q^2})^{e+e'}$ be a non-degenerate unitary space.
    If $(e,e',q)\ne(1,1,2)$, then the
    proportion of pairs $(U,U')$ of non-degenerate subspaces of dimensions
    $e$ and $e'$ respectively that satisfy $U\cap U'=0$,
    and hence $V=U\oplus U'$, is at least $1-\frac{9}{5q^2}>0$.
  \end{enumerate}
\end{theorem}

\subsection{The symplectic case}\label{ss:Sp}
In this subsection we establish the following lower bound.

\begin{theorem}\label{T:Spbound}
  Let $d,e, e'$ be positive even integers where $d\ge e+e'$.
  Let~$V=(\F_q)^d$ be a non-degenerate symplectic $d$-space over $\F_q$.
  Set $\cU=\Binom{V}{e}$ and $\cU'=\Binom{V}{e'}$. Then
  \[
    \rho(\bSp,V,\cU,\cU')
    > \left(1-\frac{1}{2q^2\log(q)}\right)\left(1-\frac{5}{3q}\right)
    > 1-\frac{7}{4q}\qquad\textup{for all $q\ge2$.}
  \]
\end{theorem}

\begin{lemma}\label{lem:symplecticbound}
    Suppose that $V=(\F_q)^d$ is a non-degenerate symplectic space, and
    $W$ is a non-degenerate $(e+e')$-dimensional subspace. Then $d,e,e'$ are
    even, $d\ge e+e'$ and
    \[
    \frac{|\Binom{V}{e+e'}^\bSp|\cdot|\Binom{W}{e}^\bSp|\cdot|\Binom{W}{e'}^\bSp|}
         {|\Binom{V}{e}^\bSp|\cdot |\Binom{V}{e'}^\bSp|}
    > 1- \frac{1}{2q^2 \log(q)}.
    \]
\end{lemma}

\begin{proof}
  For the duration of this proof and, for  $e$ even,
  we abbreviate $\omega(e/2,q^2)$ by $\omega(e)$.  Repeated use
  of Proposition~\ref{prop:pA}(b) and extensive cancellation shows
  \begin{align*}
  \frac{|\Binom{V}{e+e'}|\cdot |\Binom{W}{e}|\cdot |\Binom{W}{e'}|}
       {|\Binom{V}{e}|\cdot |\Binom{V}{e'}|}
     &=
     \frac{q^{(e+e')(d-e-e')}\omega(d) }{\omega(e{+}e')\omega(d{-}e{-}e')}\cdot
     \frac{ q^{2ee'}\omega(e{+}e')^2 }{\omega(e)^2\omega(e')^2}\cdot
     \frac{\omega(e) \omega(d{-}e) }{ q^{(d-e)e}\omega(d)} \cdot
     \frac{\omega(e') \omega(d{-}e') }{ q^{(d-e')e'}\omega(d)} \\
     &=
     \frac{\omega(d)}{\omega(e{+}e')\omega(d{-}e{-}e')}\cdot
     \frac{\omega(e{+}e')^2}{\omega(e)^2\omega(e')^2}\cdot
     \frac{\omega(e) \omega(d{-}e)}{\omega(d)} \cdot
     \frac{\omega(e') \omega(d{-}e') }{\omega(d)} \\
     &=
     \frac{\omega(e+e')}{\omega(e)\omega(e')} \cdot
     \frac{\omega(d-e) \omega(d-e') }{\omega(d-e-e') \omega(d)}.
  \end{align*}
  If $d=e+e'$, then this ratio is 1 and the bound holds trivially.
  Suppose that $d>e+e'$. Then $\frac{\omega(e+e')}{\omega(e)\omega(e')} \ge 1$
  by Lemma~\ref{lem:binombounds}, and applying Lemma~\ref{lem:omegabound}
  with $q$ replaced by $q^2$ gives
  \[
  \frac{\omega(d-e) \omega(d-e')}{\omega(d-e-e')\omega(d)}
  = \frac{\omega((d-e)/2,q^2) \omega((d-e')/2,q^2)}
         {\omega((d-e-e')/2,q^2)\omega(d/2,q^2)}
  > 1- \frac{1}{2q^2 \log(q)}.\qedhere
  \]
\end{proof}

\begin{proof}[Proof of Theorem~$\ref{T:Spbound}$]
  Let $W$ be a non-degenerate $(e+e')$-subspace of $V$. Set
  $\cW\coloneq\Binom{W}{e}$ and
  $\cW'\coloneq\Binom{W}{e'}$.
  It follows from Proposition~\ref{p:reduce}(a) that 
  \[
    \rho\coloneq\rho(\bSp,V,\cU,\cU')
    = \frac{|\Binom{V}{e+e'}^\bSp|\cdot |\Binom{W}{e}^\bSp|
    \cdot|\Binom{W}{e'}^\bSp|}{|\Binom{V}{e}^\bSp|\cdot |\Binom{V}{e'}^\bSp|}
    \cdot\rho(\bSp,W,\cW,\cW').
  \]
  However, $\rho(\bSp,W,\cW,\cW')\ge1-\frac{5}{3q}$ by Theorem~\ref{T:Forms}(a).
  Therefore Lemma~\ref{lem:symplecticbound} implies that
  $\rho > \left(1-\frac{1}{2q^2\log(q)}\right)\left(1-\frac{5}{3q}\right)$.
  This proves the first inequality. For the second inequality,
  we write $\alpha=\frac{5}{3}$  and $\beta=\frac{1}{2q\log(q)}$ and aim to
  show that $(1-\frac{\alpha}{q})(1-\frac{\beta}{q})\ge1-\frac{7}{4q}$.
  This is equivalent to $\frac{7}{4} > \alpha+\beta-\frac{\alpha\beta}{q}
  =\frac{5}{3}+\frac{1}{2q\log(q)}-\frac{5}{6q^2\log(q)}$. The latter
  is true for all $q\ge2$.
\end{proof}

\subsection{The unitary case}\label{ss:U}

In this subsection we verify the following lower bound.

\begin{theorem}\label{T:Ubound}
  Suppose that $d,e, e'$ are positive integers where $d\ge e+e'$.
  Let~$V=(\F_{q^2})^d$ be a non-degenerate hermitian $d$-space.
  Set $\cU=\Binom{V}{e}$ and $\cU'=\Binom{V}{e'}$. If $q\ge2$, then
  \[
    \rho(\bU,V,\cU,\cU')
    \ge 1-\frac{1.72}{q}.  
  \]
\end{theorem}

\begin{lemma}\label{lem:unitarybound}
  Suppose that $V=(\F_{q^2})^d$ is a non-degenerate hermitian space, and
  $W$ is a non-degenerate $(e+e')$-dimensional subspace. Then $d\ge e+e'$ and
  \[
    \frac{|\Binom{V}{e+e'}^\bU|\cdot |\Binom{W}{e}^\bU|\cdot |\Binom{W}{e'}^\bU|}
         {|\Binom{V}{e}^\bU|\cdot |\Binom{V}{e'}^\bU|}
    \ge \frac{(1-q^{-1})(1-q^{-2}) }{1+q^{-2}}.
   \] 
\end{lemma}

\begin{proof}
  The inequality $d\ge e+e'$ is clear.
  For the duration of this proof we abbreviate
  $\omega(e,-q)$ by $\omega(e)$ and we suppress the superscript ${}^\bU$
  in notation such as $\Binom{V}{e}^\bU$. Repeated use
  of Proposition~\ref{prop:pA}(a) and extensive cancellation shows
  \begin{align*}
    \frac{|\Binom{V}{e+e'}|\cdot |\Binom{W}{e}|\cdot |\Binom{W}{e'}|}
       {|\Binom{V}{e}|\cdot |\Binom{V}{e'}|}
       &= \frac{q^{2(e+e')(d-e-e')}\omega(d)}
               {\omega(e{+}e')\omega(d{-}e{-}e')}\cdot
       \frac{q^{4ee'}\omega(e+e')^2}
            {\omega(e)^2\omega(e')^2}\cdot
       \frac{\omega(e) \omega(d{-}e)}
            {q^{2(d-e)e}\omega(d)} \cdot
       \frac{\omega(e') \omega(d{-}e')}{  q^{2(d-e')e'}\omega(d)} \\
       &=
       \frac{\omega(e+e')}{\omega(e)\omega(e')} \cdot
       \frac{\omega(d-e) \omega(d-e')}{\omega(d-e-e')\omega(d)}.
  \end{align*}
  If $d=e+e'$, then this ratio is 1 and the bound holds trivially.
  Suppose that $d>e+e'$. Then
  $\frac{\omega(e+e')}{\omega(e)\omega(e')}
  \ge \frac{1-(-q)^{-(e+e')}}{1+q^{-1}}
  \ge \frac{1-q^{-2}}{1+q^{-1}}$
  by Lemma~\ref{lem:binombounds}. Applying the upper and
  lower bounds from Lemma~\ref{lem:int}(b) gives
  \[
    \frac{\omega(d-e) \omega(d-e')}{\omega(d-e-e')  \omega(d)}
    \ge \frac{1-q^{-2}}{1+q^{-2}}.
  \]
  In summary, we have
  $\frac{|\Binom{V}{e+e'}|\cdot |\Binom{W}{e}|\cdot |\Binom{W}{e'}|}
        {|\Binom{V}{e}|\cdot |\Binom{V}{e'}|}
  \ge \frac{1-q^{-2}}{1+q^{-1}}\cdot \frac{1-q^{-2}}{1+q^{-2}}
  =\frac{(1-q^{-1})(1-q^{-2}) }{1+q^{-2}}$ as claimed.
\end{proof}

\begin{proof}[Proof of Theorem~$\ref{T:Ubound}$]
  The bound holds when $d=e+e'$ as~\cite{Forms}*{Theorem~4.1} shows that
  \[
    \rho(\bU,V,\cU,\cU')\ge 1-\frac{2}{q^2}\ge1-\frac{1}{q}>1-\frac{1.72}{q}.
  \]
    
  Suppose, henceforth that $d>e+e'$.  
  Let $W$ be a non-degenerate $(e+e')$-subspace of $V$.
  We suppress the superscript ${}^\bU$ in notation such as $\Binom{V}{e}^\bU$.
  Set $\cW\coloneq\Binom{W}{e}$ and $\cW'\coloneq\Binom{W}{e'}$.
  It follows from Proposition~\ref{p:reduce}(a) that
  \begin{equation}\label{E:Ox}
    \rho(\bU,V,\cU,\cU') = \frac{|\Binom{V}{e+e'}|\cdot |\Binom{W}{e}|
      \cdot|\Binom{W}{e'}|}{|\Binom{V}{e}|\cdot |\Binom{V}{e'}|}
    \cdot\rho(\bU,W,\cW,\cW').
  \end{equation}

  Suppose first that $(e,e',q)=(1,1,2)$. It follows from the first
  displayed formula in the proof of Lemma~\ref{lem:unitarybound} that
  \begin{align*}
    \frac{|\Binom{V}{e+e'}|\cdot |\Binom{W}{e}|
      \cdot|\Binom{W}{e'}|}{|\Binom{V}{e}|\cdot |\Binom{V}{e'}|}
    &=\frac{\omega(2,-q)}{\omega(1,-q)^2} \cdot
    \frac{\omega(d-1,-q)^2}{\omega(d-2,-q)\omega(d,-q)}
    =\frac{1-q^{-2}}{1+q^{-1}}\cdot
    \frac{1-(-q)^{-d+1}}{1-(-q)^{-d}}\\
    &=\frac{(1-q^{-1})(1-(-q)^{-d+1})}{1-(-q)^{-d}}
    \ge \frac{(1-q^{-1})(1-q^{-2})}{1+q^{-3}}
    =\frac{1}{3}.    
  \end{align*}
  However $\rho(\bU,W,\cW,\cW')\ge1-\frac{2}{q^2}=\frac{1}{2}$
  by~\cite{Forms}*{Theorem~4.1}. The result now follows from~\eqref{E:Ox} as
  $\rho(\bU,V,\cU,\cU')\ge\frac{1}{3}\cdot\frac{1}{2}
  =\frac{1}{6}>1-\frac{1.72}{2}$.

  Suppose next that $(e,e',q)\ne(1,1,2)$. 
  Then $\rho(\bU,W,\cW,\cW')\ge1-\frac{9}{5q^2}$ by~Theorem~\ref{T:Forms}(b).
  Thus~\eqref{E:Ox} and Lemma~\ref{lem:unitarybound} show
  $\rho(\bU,V,\cU,\cU')\ge\frac{(1-q^{-1})(1-q^{-2})(1-\frac{9}{5}q^{-2}) }{1+q^{-2}}$.
  To prove the stated bound, we must prove that
  $\frac{(1-q^{-1})(1-q^{-2})(1-\frac{9}{5}q^{-2}) }{1+q^{-2}}\ge1-\frac{1.72}{q}$.
  This, in turn, is equivalent to showing
  \[
  1.72
  \ge\frac{q\left[(1+q^{-2})-(1-q^{-1})(1-q^{-2})(1-\frac{9}{5}q^{-2})\right]}{1+q^{-2}}
  =\frac{1+\frac{19}{5} q^{-1}-\frac{14}{5}q^{-2}-\frac{9}{5}q^{-3}+\frac{9}{5}q^{-4}}{1+q^{-2}}.
  \]
  holds for all $q\ge2$.
  The right side approaches 1 quite rapidly as $q\to\infty$, so the upper
  bound holds
  if~$q$ is `large'. It is not hard to see that the right side attains
  a maximum value of $\frac{43}{25}=1.72$ when $q=3$. This completes the proof.
\end{proof}

\subsection{The orthogonal case}\label{ss:O}

We use Theorem~1.1 of~\cite{Forms} which we paraphrase below in the context
we need.

\begin{theorem}\label{T:O}
  Suppose that $e,e'$ are positive even integers and $W=(\F_q)^{e+e'}$ is
  a non-degenerate orthogonal space of type $\pm$.  If $q\ge3$, then
  the proportion of pairs $(U,U')$ where $U$ is a
  non-degenerate $e$-subspace of $W$ of type~$\sigma$ and $U'$ is a
  non-degenerate $e'$-subspace of $W$ of type~$\sigma'$, which satisfy
  $U\cap U'=0$, is at least  $1-c/q$  where $c=43/16$.
\end{theorem}

Our goal in this subsection is to prove the following theorem.

\begin{theorem}\label{T:orthobound}
  Suppose that $e, e'$ are even positive integers and $d\ge e+e'$. Let
  $V$ be a non-degenerate orthogonal $d$-space of type $\eps$ over $\F_q$, where
  $\eps \in \{+,\circ,-\}.$ Let $\sigma,\sigma'\in\{-,+\}$ and let
  $\cU=\Binom{V}{e}^\eps_\sigma$, $\cU'=\Binom{V}{e'}^\eps_{\sigma'}$ be the set
  of non-degenerate subspaces of dimensions $e,e'$ of types $\sigma,\sigma'$,
  respectively. Then, if $q\ge3$, the proportion of pairs
  $(U,U')\in\cU\times\cU'$ for which $U\cap U'=0$ and $U+U'$ is non-degenerate
  satisfies
  \[
  \rho(\bO^\eps,V,\cU,\cU')\ge1-\frac{2.85}{q}=0.05\quad\textup{if $q=3$, and}
  \quad
  \rho(\bO^\eps,V,\cU,\cU')\ge1-\frac{25}{8q}\quad\textup{if $q\ge4$.}
  \]
\end{theorem}

\begin{remark}\label{R:O}
As discussed in Section~\ref{subsec:stingelts}, our algorithmic application finds group elements $g,g'$ in the isometry
group of $V$ where  $g$ acts irreducibly on~$U\coloneq \im(g-1)$ and $g'$
acts irreducibly on $U'\coloneq \im(g'-1)$. 
In the orthogonal case $e=\dim(U)\ge2$ implies that $e$
must be even and the type $\sigma$ of $U$ is \emph{minus},
see Lemma~\ref{l:prop1}(b). Therefore in the algorithmic application
the random elements $g,g'$ correspond to non-degenerate subspaces $U$ 
and $U'$ of even dimension and
minus type. By contrast, in Theorem~\ref{t:main1} (and Theorem~\ref{T:orthobound}) 
the non-degenerate
subspaces $U,U'$ may have types $\sigma,\sigma'\in\{-,+\}$ as they
need not arise from group elements $g,g'\in\GO(V)$ as in our algorithmic application.\hfill$\diamond$
\end{remark}

Let $V=(\F_q)^d$ be a non-degenerate orthogonal $d$-space of type $\eps$. As
$e,e'$ are even, so is $e+e'$ and for these dimensions
the subspace type equals the intrinsic type,
so we can unambiguously abbreviate to type.
Let $\cU=\Binom{V}{e}^\eps_\sigma$ denote the set of non-degenerate $e$-subspaces
of $V$ of type $\sigma$, and similarly $\cU'=\Binom{V}{e'}^\eps_{\sigma'}$. 
We note first that Theorem~\ref{T:orthobound} follows   immediately
from Theorem~\ref{T:O} if $d=e+e'$, since in this case Theorem~\ref{T:O}
implies that, for all $q\ge3$, 
\[
  \rho(\bO^\eps,V,\cU,\cU') \ge 1 -\frac{43}{16q}
  > \max\left\{1- \frac{2.85}{q}, 1-\frac{25}{8q}\right\}.
\]

Thus we may, and shall, assume that $d\ge e+e'+1$, so that
$\Binom{V}{e+e'}^\eps_\tau$ is non-empty for each $\tau\in\{-,+\}$.
Fix $W_\tau\in\Binom{V}{e+e'}^\eps_\tau$ for $\tau\in\{-,+\}$. If $q\ge3$, then
it follows from Theorem~\ref{T:O} and Proposition~\ref{p:reduce}(b) that
\[
\rho(\bO^\eps,V,\cU,\cU') \ge\left(1-\frac{43}{16q}\right)
\sum_{\tau\in\{+,-\}}
\frac{|{\Binom{V}{e+e'}}^\eps_\tau|\cdot
  |{\Binom{W_\tau}{e}}^\tau_{\sigma}|\cdot
  |{\Binom{W_\tau}{e'}}^\tau_{\sigma'}|}
     {|{\Binom{V}{e}}^\eps_{\sigma}|\cdot |{\Binom{V}{e'}}^\eps_{\sigma'}|}.
\]

Arguing as in the proof of Lemma~\ref{lem:symplecticbound} we
temporarily set $\omega(n)=\omega(\lfloor n/2\rfloor,q^2)$.
By Proposition~\ref{prop:pA}(c), for $e$ even,
the number of $e$-subspaces of $V$ of type $\sigma$ equals
\begin{align}
  \left|\Binom{V}{e}^\eps_\sigma\right|
  &=q^{e(d-e)}\kappa( e,d-e, \eps, \sigma)
  \frac{\omega(d)}{\omega(e)\omega(d-e)}
       \quad\textup{where} \label{E:Nsub}\\
    \kappa( e, d-e, \eps, \sigma) &=
  \frac{(1+\sigma q^{-e/2})(1+\eps \sigma q^{-\lfloor d/2\rfloor+e/2})}{2(1+\eps  q^{-\lfloor d/2\rfloor})}.\label{E:kappa}
\end{align}

Using the formula~\eqref{E:Nsub} for the other factors in the product
above, we obtain a complicated expression for
$\frac{|{\Binom{V}{e+e'}}^\eps_\tau|\cdot
  |{\Binom{W_\tau}{e}}^\tau_\sigma|\cdot |{\Binom{W_\tau}{e'}}^\tau_{\sigma'}|}
    {|{\Binom{V}{e}}^\eps_\sigma|\cdot
      |{\Binom{V}{e'}}^\eps_{\sigma'}|}$ which admits two simplifications.
  First, the powers of $q$ all cancel, and second the following factor which
  is independent of $\tau\in\{+,-\}$ can be estimated:
  \[
  \frac{\omega(e+e')}{\omega(e)\omega(e')} \ge1\quad\mbox{and}\quad 
  \frac{\omega(d-e) \omega(d-e')}{\omega(d-e-e')\omega(d)}
    > 1- \frac{1}{2q^2 \log(q)}
  \]
  where the first inequality follows from Lemma~\ref{lem:binombounds}, 
  and the second follows from applying Lemma~\ref{lem:omegabound} with $q$ replaced by $q^2$ (as in the last line of the proof of Lemma~\ref{lem:symplecticbound}).   
  Using  the above lower bound gives
  the following (strict) lower bound for $\rho(\bO^\eps,V,\cU,\cU')$:
  \[
\left(1-\frac{43}{16q}\right)\left(1- \frac{1}{2q^2 \log(q)}\right)
\sum_{\tau\in\{+,-\}}\frac{
    \kappa(e+e', d-e-e', \eps, \tau)
    \kappa( e, e', \tau, \sigma)
    \kappa( e', e, \tau, \sigma') 
  }{\kappa( e, d-e, \eps, \sigma)
    \kappa( e', d-e', \eps, \sigma')}.
\]
The following technical lemma helps us bound the above sum.

\begin{lemma}\label{lem:abcd}
  Let $\alpha,\beta,\gamma,\delta\in \R$ such that
  $\delta \neq \pm 1.$ Then
  \[
  \frac12\sum_{\tau\in\{-,+\}}
    \frac{(1-\tau \alpha)(1- \tau \beta)(1-\tau \gamma)}{1+\tau \delta}
  = \frac{ 1 + \alpha\beta + \alpha\gamma + \alpha\delta + \beta\gamma +\beta\delta +\gamma\delta + \alpha\beta\gamma\delta}{1-\delta^2}.
  \]
\end{lemma}

\begin{proof}
  Let $\Lambda$ be the stated sum. Adding fractions and cancelling ``odd''
  terms gives
  \begin{align*}
    \Lambda&=\frac{(1+\alpha)(1+\beta)(1+\gamma)(1+\delta)+(1-\alpha)(1-\beta)(1-\gamma)(1-\delta)}{2(1-\delta)(1+\delta)}\\
    &=\frac{ 1 + \alpha\beta + \alpha\gamma + \alpha\delta + \beta\gamma +\beta\delta +\gamma\delta + \alpha\beta\gamma\delta}{1-\delta^2}.\qedhere
  \end{align*}
\end{proof}

\begin{lemma}\label{lem:kappabound}
  Suppose that $d \ge e+e'+1$ where $e,e'$ are both even and $d$ may be odd.
  Suppose further that the types
  $\sigma,\sigma',\eps$ satisfy $\sigma,\sigma'\in \{-,+\}$ and
  $\eps \in \{-,\circ,+\}$. Then
  \[
    K\coloneq\sum_{\tau\in\{-,+\}} \frac{
    \kappa(e+e', d-e-e', \eps, \tau)
    \kappa( e, e', \tau,\sigma)
    \kappa( e', e, \tau, \sigma') 
  }{\kappa( e, d-e, \eps, \sigma)
    \kappa( e', d-e', \eps, \sigma')}
  \]
  satisfies
  $K\ge\frac{(1 - q^{-3})(1 -2 q^{-2}-3q^{-3}-q^{-5})}{(1 + q^{-2})^2}$ and hence
  $K\ge1-\frac{9}{2q^{2}}$ for $q\ge3$.
\end{lemma}

\begin{proof}
Let $K$ be as in the statement.
Using the definition~\eqref{E:kappa} of $\kappa$ and
cancelling (8 of the 15 factors) gives
\begin{align*}
 K& =
  \frac{1}{2}
  \sum_{\tau\in\{-,+\}}
  \frac{
    (1 + \eps \tau q^{-\lfloor \frac{d}{2}\rfloor+\frac{e+e'}{2}})
    (1 + \tau\sigma q^{-\frac{e'}{2}})
    (1 + \tau \sigma' q^{-\frac{e}{2}})      (1 + \eps q^{-\lfloor \frac{d}{2}\rfloor})}
    {(1 + \tau q^{-\frac{e+e'}{2}}) (1+\eps \sigma q^{-\lfloor \frac{d}{2}\rfloor+\frac e2})
         (1 + \eps\sigma' q^{-\lfloor\frac{d}{2}\rfloor+\frac{e'}{2}})}\\
 & =
       \frac{1 + \eps q^{-\lfloor \frac{d}{2}\rfloor}}
            {(1 + \eps\sigma q^{-\lfloor\frac{d}{2}\rfloor+\frac e2})
              (1 + \eps\sigma' q^{-\lfloor\frac{d}{2}\rfloor+\frac {e'}{2}})}
            \cdot \Lambda,
\end{align*}
where
\begin{equation}\label{E:Lambda}
    \Lambda = \frac{1}{2} \sum_{\tau\in\{-,+\}}
  \frac{(1 + \eps \tau q^{-\lfloor \frac{d}{2}\rfloor+\frac{e+e'}{2}})
        (1 + \tau\sigma q^{-\frac{e'}{2}})
        (1 + \tau \sigma' q^{-\frac{e}{2}})}
    {1 + \tau q^{-\frac{e+e'}{2}}}.
\end{equation}
Suppose that $d\ge e+e'+2$.
This implies $\lfloor\frac{d}{2}\rfloor>\frac{e+e'}{2}\ge2$,
whence $-\lfloor\frac{d}{2}\rfloor+\frac{e}{2}<-\frac{e'}{2}\le-1$,
and so $-\lfloor\frac{d}{2}\rfloor+\frac{e}{2}\le-2$;
and similarly $-\lfloor\frac{d}{2}\rfloor+\frac{e'}{2}\le-2$.
Using Lemma~\ref{lem:abcd} gives
\def\kp{{\kern-0.05pt+\kern-0.05pt}}
\def\km{{\kern-0.05pt-\kern-0.05pt}}
\begin{align*}
\Lambda &= 
  \frac{
    1 \kp \eps \sigma q^{-\lfloor \frac{d}{2}\rfloor+\frac{e}{2}} \kp \eps \sigma' q^{-\lfloor \frac{d}{2}\rfloor+\frac{e'}{2}}
    \km \eps q^{-\lfloor \frac{d}{2}\rfloor}
    \kp \sigma \sigma' q^{-\frac{e+e'}{2}}
    \km \sigma  q^{-\frac{e}{2}-e'}
    \km \sigma'  q^{-\frac{e'}{2}-e}
    \km \eps \sigma \sigma' q^{-\lfloor \frac{d}{2}\rfloor - \frac{e+e'}{2}}}
       { 1 -  q^{-e-e'}}\\
      &\ge 1-q^{-2}-q^{-2} -q^{-3} -q^{-2} -q^{-3} -q^{-3} -q^{-5}\\
      & = 1 -3 q^{-2} -3q^{-3} -q^{-5}.
\end{align*}
Upon closer inspection, $\eps\sigma$, $\eps\sigma'$ and
$\sigma\sigma'$ cannot all be $-1$, so that the sharper bound
$\Lambda\ge1 -2 q^{-2} -3q^{-3} -q^{-5}$ holds.

Inserting this lower bound for $\Lambda$ into the expression for $K$
we obtain
\begin{align*}
  K& \ge \frac{1 + \eps q^{-\lfloor \frac{d}{2}\rfloor}}
           {(1 + \eps\sigma q^{-\lfloor\frac{d}{2}\rfloor+\frac{e}{2}})
            (1 + \eps\sigma' q^{-\lfloor\frac{d}{2}\rfloor+\frac{e'}{2}})}\cdot
       \frac{1 -2 q^{-2} -3q^{-3} -q^{-5}}{1}\\
  &\ge \frac{(1 - q^{-3})(1 -2 q^{-2}-3q^{-3}-q^{-5})}{(1 + q^{-2})^2}.
\end{align*}
This proves the first bound when $d\ge e+e'+2$.
To prove that this lower bound is at least $1-\frac{9}{2q^2}$
for all $q\ge3$, we rearrange this inequality and show that it is equivalent to
proving that the following is true for all $q\ge3$:
\begin{align*}
  \frac{9}{2}&\ge\frac{q^2((1 + q^{-2})^2-(1 - q^{-3})(1 -2 q^{-2}-3q^{-3}-q^{-5}))}
          {(1 + q^{-2})^2}\\
   &=\frac{4+4q^{-1}+q^{-2}-q^{-3}-3q^{-4}-q^{-6}}{(1 + q^{-2})^2}.
\end{align*}
As $q\to\infty$, the the right side  $\to 4$. Indeed the values of
the right side lie between $4$ and $\frac{9}2$ for all $q\ge3$. This proves
the second bound  when $d\ge e+e'+2$.

Now we consider the  remaining case where $d=e+e'+1$. Here $d$ is odd so $\lfloor\frac{d}{2}\rfloor=\frac{e+e'}{2}$.
If $\tau = -\eps$, then one summand in the expression for $\Lambda$
in~\eqref{E:Lambda} is zero because  the factor
$1 + \eps \tau q^{-\lfloor\frac{d}{2}\rfloor+\frac{e+e'}{2}}$ is zero.
Hence 
\begin{align*}
  K
   & =
       \frac{1 + \eps q^{-\lfloor \frac{d}{2}\rfloor}}
            {2(1 + \eps\sigma q^{-\lfloor\frac{d}{2}\rfloor+\frac{e}{2}})(1 + \eps\sigma' q^{-\lfloor\frac{d}{2}\rfloor+\frac{e'}{2}})}\cdot
  \frac{
    (1 +  q^{-\lfloor \frac{d}{2}\rfloor+\frac{e+e'}{2}})
    (1 + \eps\sigma q^{-\frac{e'}{2}})
    (1 + \eps \sigma' q^{-\frac{e}{2}})}
       {    (1 + \eps q^{-\frac{e+e'}{2}})}\\
&  =
       \frac{1 + \eps q^{-\lfloor \frac{e+e'}{2}\rfloor}}
            {2(1 + \eps\sigma q^{-\frac{e'}{2}})(1 + \eps\sigma' q^{-\frac{e}{2}})}
            \cdot
  \frac{(1+q^0)(1 + \eps\sigma q^{-\frac{e'}{2}})(1 + \eps \sigma' q^{-\frac{e}{2}})}
    {1 + \eps q^{-\lfloor \frac{e+e'}{2}\rfloor}}.
\end{align*}
Hence $K=1$ and both of the stated bounds also hold in this final case.
\end{proof}

\begin{proof}[Proof of Theorem~$\ref{T:orthobound}$]
	As discussed after Remark~\ref{R:O},  Theorem~$\ref{T:orthobound}$ follows from Theorem~\ref{T:O} if $d=e+e'$, so assume that $d\ge e+e'+1$.
  Let $K$ be as in Lemma~\ref{lem:kappabound}. It follows from the display
  before Lemma~\ref{lem:abcd}  that
  \begin{align*}
    \rho(\bO^\eps,V,\cU,\cU') 
         &>\left(1-\frac{43}{16q}\right)
         \left(1-\frac{1}{2q^2\log(q)}\right)K.
  \end{align*}
  Setting $q=3$ in the first lower bound in Lemma~\ref{lem:kappabound} shows that
  $\rho(\bO^\eps,V,\cU,\cU')>0.05$. Thus setting $c=2.85$, we have
  $\rho>0.05=1-\frac{c}{3}$.
  Similarly, if $q\in\{4,5,7,8,9\}$ then again using the first lower bound
  on $K$, we find that $\rho(\bO^\eps,V,\cU,\cU')>1-\frac{3.125}{q}$ holds.
  For $q\ge11$ we use
  the second lower bound $K\ge 1-\frac{9}{2q^2}$ in Lemma~\ref{lem:kappabound}.
  This shows
  \[
  \rho(\bO^\eps,V,\cU,\cU')>\left(1-\frac{43}{16q}\right)\left(1-\frac{1}{2q^2\log(q)}\right)\left(1-\frac{9}{2q^2}\right)
  =\left(1-\frac{\alpha}{q}\right)\left(1-\frac{\beta}{q}\right)\left(1-\frac{\gamma}{q}\right)
  \]
  where $\alpha=\frac{43}{16}$, $\beta=\frac{1}{2q\log(q)}$ and
  $\gamma=\frac{9}{2q}$. The inequality
  $(1-\frac{\alpha}{q})(1-\frac{\beta}{q})(1-\frac{\gamma}{q})\ge1-\frac{3.125}{q}$
  is equivalent to
  \[
  3.125\ge q\left(1-\left(1-\frac{\alpha}{q}\right)\left(1-\frac{\beta}{q}\right)\left(1-\frac{\gamma}{q}\right)
  \right)
  =\alpha+\beta+\gamma-\frac{\alpha\beta+\beta\gamma+\gamma\alpha}{q}+
  \frac{\alpha\beta\gamma}{q^2}.
  \]
  The stronger condition
  $3.125\ge\alpha+\beta+\gamma+ \frac{\alpha\beta\gamma}{q^2}$ does indeed
  hold for all $q\ge11$.
  Hence $\rho(\bO^\eps,V,\cU,\cU')\ge 1-\frac{3.125}{q}$ holds for all $q\ge4$.
\end{proof}

\end{document}